\documentclass[a4paper,11pt]{amsart}
\usepackage[colorlinks, linkcolor=blue,anchorcolor=Periwinkle,
    citecolor=blue,urlcolor=Emerald]{hyperref}
\usepackage[all]{xy}
\SelectTips{cm}{}

\usepackage{graphicx}
\usepackage{psfrag}

\usepackage{mathtools}
\usepackage{tikz}
\usepackage{tikz-cd}
\usepackage{extarrows}
\usepackage{adjustbox}

\usepackage{amssymb}
\usetikzlibrary{decorations.pathreplacing}
\usetikzlibrary{matrix,arrows}

\usepackage{enumitem}

\usepackage{stmaryrd}

\newcommand{\ie}{{i.e.}\ }
\newcommand{\cf}{{cf.}\ }

\newcommand{\ko}{\: , \;}

\newcommand{\ol}[1]{\overline{#1}}

\numberwithin{equation}{subsection}

\newtheorem{classification-theorem}[subsection]{Classification Theorem}
\newtheorem{decomposition-theorem}[subsection]{Decomposition Theorem}
\newtheorem{proposition-definition}[subsection]{Proposition-Definition}
\newtheorem{periodicity-conjecture}[subsection]{Periodicity Conjecture}

\newtheorem{theorem}{Theorem}
\numberwithin{theorem}{subsection}
\newtheorem{thmx}{Theorem}

\newtheorem{lemma}[theorem]{Lemma}
\newtheorem{proposition}[theorem]{Proposition}
\newtheorem{corollary}[theorem]{Corollary}

\newtheorem{remark}[theorem]{Remark}
\newtheorem{remarks}[theorem]{Remarks}
\newtheorem{assumptions}[theorem]{Assumptions}

\newcommand{\reminder}[1]{}

\renewcommand{\mod}{\mathrm{mod}\,}

\newcommand{\Mod}{\mathrm{Mod}\,}

\newcommand{\per}{\mathrm{per}\,}
\newcommand{\pvd}{\mathrm{pvd}\,}

\newcommand{\op}{^{op}}

\newcommand{\HH}{\mathrm{HH}}
\newcommand{\HC}{\mathrm{HC}}
\newcommand{\HN}{\mathrm{HN}}
\newcommand{\HP}{\mathrm{HP}}

\newcommand{\Tr}{\mathrm{Tr}}
\newcommand{\tr}{\mathrm{tr}}
\newcommand{\PC}{\mathrm{PC}\,}
\newcommand{\PCAlgc}{\mathrm{PCAlgc}\,}
\newcommand{\red}{\mathrm{red}}
\newcommand{\pc}{\mathrm{pc}}

\newcommand{\res}{\mathrm{res}}
\newcommand{\sym}{\mathrm{sym}}

\newcommand{\cok}{\mathrm{cok}\,}
\newcommand{\im}{\mathrm{im}\,}
\renewcommand{\ker}{\mathrm{ker}\,}
\newcommand{\cone}{\mathrm{cone}\,}
\newcommand{\cocone}{\mathrm{cocone}\,}

\newcommand{\Z}{\mathbb{Z}}
\newcommand{\N}{\mathbb{N}}

\newcommand{\D}{\mathbb{D}}

\newcommand{\iso}{\xrightarrow{_\sim}}
\newcommand{\liso}{\xleftarrow{_\sim}}
\newcommand{\id}{\mathbf{1}}

\newcommand{\Hom}{\mathrm{Hom}}

\newcommand{\RHom}{\mathrm{RHom}}

\newcommand{\Ext}{\mathrm{Ext}}

\newcommand{\rad}{\mathrm{rad}\,}

\newcommand{\ten}{\otimes}
\newcommand{\lten}{\overset{\mathrm{L}}{\ten}}

\newcommand{\ldb}{\{\!\!\{}
\newcommand{\rdb}{\}\!\!\}}
\newcommand{\ca}{{\mathcal A}}
\newcommand{\cb}{{\mathcal B}}
\newcommand{\cc}{{\mathcal C}}
\newcommand{\cd}{{\mathcal D}}

\newcommand{\bp}{\mathbf{p}}

\newcommand{\rb}{\mathrm{B}}
\newcommand{\rc}{\mathrm{C}}

\newcommand{\La}{\Lambda}
\newcommand{\Si}{\Sigma}
\newcommand{\si}{\sigma}

\newcommand{\eps}{\varepsilon}
\newcommand{\del}{\partial}

\renewcommand{\hat}[1]{\widehat{#1}}

\renewcommand{\tilde}[1]{\widetilde{#1}}

\begin{document}

\date{\today}

\title[Relative Calabi--Yau structures and ice quivers with potential]{Relative Calabi--Yau structures and\\[0.15cm]
ice quivers with potential}

\author{Bernhard Keller}
\address{Universit\'e Paris Cit\'e and Sorbonne Université, CNRS, IMJ-PRG, F-75013 Paris, France}
\email{bernhard.keller@imj-prg.fr}
\urladdr{https://webusers.imj-prg.fr/~bernhard.keller/}

\author{Junyang Liu}
\address{Universit\'e Paris Cit\'e and Sorbonne Université, CNRS, IMJ-PRG, F-75013 Paris, France 
and Yau Mathematical Sciences Center, Tsinghua University, Beijing 100084, China}
\email{liuj@imj-prg.fr}
\urladdr{https://webusers.imj-prg.fr/~junyang.liu}

\begin{abstract}

In 2015, Van den Bergh showed that complete $3$-Calabi--Yau algebras
over an algebraically closed field of characteristic $0$ are equivalent to Ginzburg dg algebras associated with quivers with potential. He also proved the natural generalisation to higher dimensions and non-algebraically closed ground fields. The relative version of the notion of Ginzburg dg algebra is that of Ginzburg morphism. For example, every ice quiver with
potential gives rise to a Ginzburg morphism. We generalise Van den Bergh's theorem
by showing that, under suitable assumptions, any morphism with a relative Calabi--Yau structure is equivalent to a Ginzburg(--Lazaroiu) morphism. In particular, in dimension~$3$ and over an
algebraically closed ground field of characteristic $0$, it is given by an ice quiver with potential. Thanks to the work of
Bozec--Calaque--Scherotzke, this result can also be viewed as a noncommutative analogue of Joyce--Safronov's Lagrangian neighbourhood theorem in derived symplectic geometry.
\end{abstract}



%\dedicatory{Dedicated to the memory of Ragnar-Olaf Buchweitz.}

\maketitle

\vspace*{-1cm}
\tableofcontents

\section{Introduction}
Following Kontsevich, a $\Hom$-finite triangulated category is called $d$-Calabi--Yau if it admits the $d$th power
of the suspension functor as a Serre functor. The terminology is motivated by the example of the
bounded derived category of coherent sheaves on a Calabi--Yau smooth projective variety of dimension $d$. In a noncommutative setting, \linebreak $d$-Calabi--Yau categories appear as bounded derived categories of finite-dimensional modules over $d$-Calabi--Yau algebras,
a notion introduced by Ginzburg in his fundamental preprint \cite{Ginzburg06}.
Here, for each quiver(=oriented graph) with potential, he constructed a special kind of dg algebra,
now called ($3$-dimensional) Ginzburg dg algebra, and showed that it is $3$-Calabi--Yau if its homology is concentrated in degree $0$. Keller observed that Ginzburg dg algebras are always $3$-Calabi--Yau and Van den Bergh proved it in the appendix to \cite{Keller11b}. These algebras found important applications in the representation-theoretic approach to the theory of cluster algebras, \cf \cite{FominZelevinsky02, FockGoncharov06a,DerksenWeymanZelevinsky08, DerksenWeymanZelevinsky10}, via cluster categories, \cf~\cite{Amiot09, KellerYang11, Keller11b}.
Motivated by the `relation completions' which occur in this context \cite{AssemBruestleSchiffler08},
Keller \cite{Keller11b} generalised the construction of Ginzburg dg algebras to deformed Calabi--Yau completions.
In \cite{Yeung16}, Yeung introduced, more generally, deformed relative Calabi--Yau completions of dg functors and
showed that they have relative Calabi--Yau structures (\cf~below) 
for dg functors between finitely cellular dg categories. This result was
generalised to dg functors between arbitrary smooth dg categories by Bozec--Calaque--Scherotzke \cite{BozecCalaqueScherotzke24},
who also confirmed Yeung's conjecture that they are the correct noncommutative analogues of cotangent bundles.

Ginzburg conjectured in \cite{Ginzburg06} that each $3$-Calabi--Yau algebra `arising in nature' comes
from a quiver with potential but this was disproved by Davison \cite{Davison12}, who showed
that this is not the case for the group algebra of the fundamental group of a compact hyperbolic manifold
of dimension greater than one. However, Van den Bergh confirmed Ginzburg's conjecture
for {\em complete} $3$-Calabi--Yau algebras in \cite{VandenBergh15}: he showed more generally that
each complete $d$-Calabi--Yau dg algebra is weakly equivalent to a deformed dg preprojective algebra.
For example, in dimension $3$, it is given by a quiver with potential. Note that, as explained in
section~1.3 of \cite{KinjoMasuda24}, potentials are of great use in Donaldson--Thomas
theory \cite{Szendroi08, KontsevichSoibelman08, JoyceSong12, Reineke11, DavisonMeinhardt15} and
cohomological Donaldson--Thomas theory \cite{KontsevichSoibelman11, SchiffmannVasserot13, DavisonMeinhardt20}.
 
A `relative' version of the notion of Calabi--Yau structure was first sketched by To\"en in
\cite{Toen14} and then fully developed by Brav--Dyckerhoff \cite{BravDyckerhoff19, BravDyckerhoff21}.
A relative (left) Calabi--Yau structure on a dg functor is given by a class in relative negative cyclic homology whose underlying Hochschild class is non-degenerate.
It should be thought of as analogous to the datum of an
orientation on a manifold {\em with boundary}.
Many examples arise as deformed relative Calabi--Yau
completions as introduced by Yeung \cite{Yeung16}. He
advocated the idea that they should be viewed as noncommutative conormal bundles,
which was justified using Kontsevich--Rosenberg’s criterion by Bozec--Calaque--Scherotzke in
\cite{BozecCalaqueScherotzke24}. A more economical `reduced' version of the relative Calabi--Yau
completion is due to Wu \cite{Wu23a}. In particular, the ($3$-dimensional) Ginzburg morphism
associated with an {\em ice} quiver with potential arises in this way and therefore
carries a relative $3$-Calabi--Yau structure.
The relative Ginzburg dg algebra is the target of this morphism. It has been used by
Wu \cite{Wu23a, KellerWu23} to construct (additive) categorifications of large classes of
cluster algebras {\em with coefficients} \cite{FominZelevinsky07} generalising Geiss--Leclerc--Schr\"oer's approach
\cite{GeissLeclercSchroeer05, GeissLeclercSchroeer06, GeissLeclercSchroeer08a, GeissLeclercSchroeer11b,
GeissLeclercSchroeer13b} and extending earlier work by Pressland
\cite{Pressland17, Pressland17b, Pressland20, Pressland22}.

One of the key features of Brav--Dyckerhoff's notion of relative Calabi--Yau structure is a gluing
construction analogous to that in cobordism of manifolds. It was used by Christ \cite{Christ22}
to give a local-to-global construction of the Ginzburg dg algebra associated with a triangulated
surface without punctures (the corresponding quiver with potential had been known since
the work of Labardini-Fragoso \cite{Labardini09a}). Christ described the (unbounded)
derived category of the Ginzburg dg algebra via global sections of a perverse schober,
which allowed him to construct \cite{Christ21} new geometric models for its objects and morphisms
and to establish \cite{Christ22a} an unexpected connection with topological Fukaya categories.

Our aim in this article is to generalise Van den Bergh's theorem to the relative case:

\begin{thmx}[see Theorem~\ref{thm:main} for details] \label{thm:A}
Under suitable assumptions, for a morphism $f\colon B \to A$ between pseudocompact dg algebras and an integer $d\geq 2$, the following are equivalent.
\begin{itemize}
\item[i)] $f$ is weakly equivalent to a $d$-dimensional Ginzburg--Lazaroiu morphism (\cf section~\ref{ss:Ginzburg-Lazaroiu morphisms}).
\item[ii)] $f$ carries a relative $d$-Calabi--Yau structure (\cf section~\ref{ss:Calabi-Yau structures}).
\end{itemize}
\end{thmx}

For example, in dimension $3$, they are given by ice quivers with potential,
\cf Corollary~\ref{cor:dimension 3}, and in dimension $2$, they are given by ice quivers (without potential),
\cf Corollary~\ref{cor:dimension 2}. We also deduce an analogous structure theorem for certain
Calabi--Yau cospans, \cf Theorem~\ref{thm:CY-cospans}. Thanks to Bozec--Calaque--Scherotzke's theorem 
\cite{BozecCalaqueScherotzke24} linking deformed relative Calabi--Yau completions to shifted cotangent bundles 
in derived symplectic geometry, Theorem~\ref{thm:A} may also be viewed as a noncommutative analogue of
Joyce--Safronov's Lagrangian neighbourhood theorem \cite{JoyceSafronov19}.

The article is organised as follows: in section~\ref{ss:pseudo-compact objects}, we recall pseudocompact vector spaces, algebras and modules from section~3 of \cite{VandenBergh15}. In section~\ref{ss:traces and duality}, we discuss the Casimir element associated with a symmetric algebra. In section~\ref{ss:on pseudo-compact dg algebras}, we discuss the derived category of a
pseudocompact dg algebra, which can be considered as enriched over the category of vector spaces or 
that of pseudocompact vector spaces. In section~\ref{ss:morphisms between pc algebras}, we introduce augmented (non-unital) morphisms between pseudocompact dg algebras and the corresponding model category.
In section~\ref{ss:Calabi-Yau structures}, we discuss left Calabi--Yau structures on pseudocompact dg algebras and right Calabi--Yau structures on (non-pseudocompact) dg algebras and the analogous notions in the relative case and the case of a cospan. In section~\ref{ss:the necklace bracket}, we recall necklace brackets and in section~\ref{ss:Ainfty-algebras and Ainfty-modules}, $A_\infty$-algebras and $A_\infty$-modules. Section~\ref{ss:ice quivers with potential} is a reminder on the (relative) Ginzburg dg algebra (and the Ginzburg morphism) associated with an (ice) quiver with potential. In section~\ref{ss:Ginzburg-Lazaroiu morphisms}, roughly following section~9.2 of \cite{VandenBergh15}, we simultaneously generalise this setup in two directions: from dimension $3$ to arbitrary dimension greater than or equal to $2$ and from tensor algebras over products of copies of the ground field to tensor algebras over arbitrary semisimple algebras. We use the term `Ginzburg--Lazaroiu morphism' for the resulting generalisation of the notion of Ginzburg morphism. 
In section~\ref{ss:the main results}, we state the main results and in section~\ref{ss:proof from i) to ii)} and \ref{ss:proof from ii) to i)}, we prove them.

\subsection*{Acknowledgements}
The authors are grateful to Damien Calaque for pointing out the reference \cite{JoyceSafronov19}.
They thank the organisers of the ARTA 2021, the ICRA 2022 and the ARTA 2022, where the second-named author presented preliminary versions of the results of this article. They are indebted to an anonymous referee for a very careful reading of the manuscript and many helpful comments.

The second-named author is supported by the China Scholarship Council (CSC, Grant No.~202006210272) and partially supported by the National Natural Science Foundation of China (Grant No.~11971255).

\section{Notation}

The following notation is used throughout the article: we let $k$ be a field. For a $k$-vector space $V$, we denote its $k$-dual space $\Hom_k(V,k)$ by $DV$. By abuse of notation, following \cite{VandenBergh15}, we write $a=a'\ten a''$ for an element $a=\sum_i a'_i\ten a''_i$ of a tensor product. Unless we specify otherwise, algebras have units but morphisms between algebras do not necessarily preserve the units. Modules are unital right modules. We assume that $k$ acts centrally on all bimodules we consider. For a $k$-algebra $l$, we denote the category of $l$-modules by $\Mod l$ and that of finitely generated $l$-modules by $\mod l$. The internal degree of a homogeneous element $a$ in a graded vector space is denoted by $|a|$. We denote the shift functor of graded vector spaces by $\Si$ and write $s\colon A\to \Si A$ for the canonical map of degree $-1$. We use cohomological grading so that differentials are of degree $1$. For any dg algebra, we denote its differential by $d$. We write $A^e$ for the dg enveloping algebra $A\ten_k A\op$ of any dg $k$-algebra $A$. For a graded quiver $Q$, we write $kQ$ for the associated completed graded path algebra. For a graded $l$-bimodule $V$, we write $T_l V$ for the completed graded tensor algebra $\prod_{p\geq 0}V^{\ten_l p}$. The component of tensor degree $n$ of an element $\eta$ in a tensor algebra is denoted by $\eta_n$. The term `symplectic form' means `graded symplectic form' and `Lagrangian subspace' means `Lagrangian homogeneous subspace'.

\section{Preliminaries}

\subsection{Pseudocompact objects} \label{ss:pseudo-compact objects}

Following section~3 of \cite{VandenBergh15}, \cf also section~IV.3 of \cite{Gabriel62}, a {\em pseudocompact vector space} is a topological vector space $V$ which has a basis of neighbourhoods of $0$ formed by distinguished subspaces of finite codimension such that $V$ is isomorphic to the inverse limit of the system formed by the quotients $V/V'$, where $V'$ runs through open subspaces of $V$. A finite-dimensional vector space endowed with the discrete topology is a pseudocompact vector space and conversely the topology on a finite-dimensional pseudocompact vector space must be the discrete topology. Denote the category of pseudocompact $k$-vector spaces by $\PC k$. Then we have the duality $\D \colon (\Mod k)\op \iso \PC k$ which maps $V$ to its $k$-dual $DV=\Hom_k(V,k)$ endowed with the topology having a basis of neighbourhoods of $0$ formed by the kernels of the restriction maps $DV\to DV'$, where $V'$ runs through finite-dimensional subspaces of $V$. Its quasi-inverse $\D \colon (\PC k)\op \iso \Mod k$ maps $W$ to the $k$-vector space formed by the continuous $k$-linear maps from $W$ to $k$. The category $\PC k$ has a monoidal structure which is given by
\[
V\ten_k W=\D(\D W\ten_k \D V)
\]
for any $V$ and $W$ in $\PC k$. A {\em pseudocompact graded vector space} is a graded vector space, where each component is endowed with a topology making it into a pseudocompact vector space. The category of pseudocompact graded $k$-vector spaces also has a monoidal structure as follows. For any pseudocompact graded $k$-vector spaces $V$ and $W$, the component of degree $n$ of $V\ten_k W$ is given by
\[
\prod_{i+j=n}V_i\ten_k W_j \: .
\]
Following \cite{Gabriel62, VandenBergh01, KellerYang11}, a {\em pseudocompact algebra} is a topological algebra $A$ which has a basis of neighbourhoods of $0$ formed by distinguished right ideals of finite codimension such that $A$ is isomorphic to the inverse limit of the system formed by the quotients $A/I$ as an \linebreak $A$-module, where $I$ runs through open right ideals of $A$. Equivalently, it is a pseudocompact vector space endowed with a continuous multiplication. Recall from part~(1) of Lemma~3.1 of \cite{VandenBergh15} that a pseudocompact algebra also has a basis of neighbourhoods of $0$ formed by distinguished two-sided ideals of finite codimension. If we replace algebras by modules and right ideals by submodules, then we obtain the definition of {\em pseudocompact modules}. If we replace algebras by graded algebras and right ideals by graded right ideals, then we obtain the definition of {\em pseudocompact graded algebras}. A {\em pseudocompact dg (=differential graded) algebra} is a pseudocompact graded algebra endowed with a continuous differential satisfying the graded Leibniz rule. Similarly, one can define {\em pseudocompact dg modules}.

\subsection{Traces and duality} \label{ss:traces and duality}

Recall that for a finite-dimensional $k$-algebra $l$, a {\em trace} on $l$ is a $k$-linear map $\tr \colon l\to k$ such that the bilinear form $l\times l \to k$ which maps $(a,b)$ to $\tr(ab)$ is symmetric and non-degenerate. Equivalently, the map from $l$ to its $k$-dual $Dl=\Hom_k(l,k)$ which maps $a$ to $\tr(a\cdot\, ?)$ is an $l$-bimodule isomorphism. Thus, if $l$ admits a trace, it is unique up to multiplication by an invertible central element of $l$. Let $l$ be a finite-dimensional $k$-algebra and $\tr$ a trace on $l$. We have the canonical $k$-linear bijection $ l\ten_k Dl \iso \Hom_k(l, l)$. 
Thus, we obtain the bijections
\[
\begin{tikzcd}
l\ten_k l \arrow{r}{\sim} & l\ten_k Dl \arrow{r}{\sim} & \Hom_k(l, l)\: .
\end{tikzcd}
\]
The {\em Casimir element} corresponding to $\tr$ is the preimage $\si$
of the identity $\id_l$ under the composed bijection. Explicitly, we can
write $\si=\sum e_i \ten f_i$, where $(e_i)$ is a $k$-basis of $l$ and $(f_i)$ the dual $k$-basis
with respect to the non-degenerate bilinear form which maps $(a,b)$ to $\tr(ab)$. Recall that following \cite{VandenBergh15}, by abuse of notation we write $\si=\si'\ten \si''$.

\begin{lemma} \label{lemma:bimodule isomorphism}
The composed bijection
\[
\begin{tikzcd}
 l\ten_k l \arrow{r}{\sim} &  l\ten_k Dl \arrow{r}{\sim} & \Hom_k(l, l)
\end{tikzcd}
\]
is an isomorphism of $l$-bimodules, where the bimodule structure on $l\ten_k l$ is given by the outer $l$-bimodule structure and the bimodule structure on $\Hom_k(l, l)$ is given by the left $l$-module structures on both arguments.
\end{lemma}

We leave the straightforward proof to the reader.

\begin{lemma} \label{lemma:Casimir element}
The Casimir element $\si$ is symmetric and $l$-central, \ie we have
\[
\si'\ten \si'' = \si''\ten \si' \quad\mbox{and} \quad
a\si'\ten \si'' = \si'\ten \si''a
\]
for all $a$ lying in $l$.
\end{lemma}

\begin{proof}
We have $\si=\sum e_i \ten f_i$, where $(e_i)$ is any basis of $l$ and $(f_i)$
the dual basis with respect to the bilinear form which maps $(a,b)$ to $\tr(ab)$. 
Since this bilinear form is symmetric, the basis $(e_i)$ is also the dual basis
of $(f_i)$ so that we also have $\si=\sum f_i \ten e_i$, 
which shows the first equality. The second equality
is clear by Lemma~\ref{lemma:bimodule isomorphism} and the fact that $\id_l$ is central in the $l$-bimodule $\Hom_k(l, l)$. 
\end{proof}

\begin{lemma} \label{lemma:l-dual}
For any right $l$-module $M$, we have the isomorphism
\[
\Hom_l(M,l) \simeq \Hom_k(M,k)
\]
of left $l$-modules mapping $\theta$ to $\tr\circ \theta$. Its inverse maps a $k$-linear form $\phi$ to the morphism mapping $m$ to $\si'\phi(m \si'')$.
\end{lemma}

\begin{proof}
We have the chain of isomorphisms
\[
\begin{tikzcd}
\Hom_l(M,l) \arrow{r}{\sim} & \Hom_l(M, \Hom_k(l,k)) \arrow[no head]{r}{\sim} & \Hom_k(M\ten_l l, k) & \Hom_k(M,k)\arrow[swap]{l}{\sim}\: ,
\end{tikzcd}
\]
where the first one being given by the composition with the isomorphism
$l \iso \Hom_k(l,k)$ mapping $a$ to $\tr(a\cdot\, ?)$. Their composition clearly
maps $\theta$ to $\tr\circ \theta$. The inverse of $a \mapsto \tr(a\cdot\, ?)$ maps
a linear form $\phi$ on $l$ to $\si'\phi(\si'')$. This implies
the second claim.
\end{proof}

For a (pseudocompact) $l$-bimodule $U$, we define $U^l$ to be the subspace of $l$-central elements in $U$ and $U_l$ to be the quotient $U/[l,U]$ by the subspace generated by the commutators (in the category of pseudocompact vector spaces). Recall that there exists a trace form on any semisimple $k$-algebra by Proposition~5 of \cite{EilenbergNakayama55}, \cf also Proposition~9.8 of \cite{CurtisReiner81}.

\begin{proposition} \label{prop:invariant space}
Suppose that the finite-dimensional $k$-algebra $l$ is separable. Then we have the $k$-linear bijection $U_l \iso U^l$ mapping $m$ to $\si' m\, \si''$.
\end{proposition}
\begin{proof}
Since $l$ is separable, by definition, it is a finitely generated projective
module over $l^e$ so that the composed $k$-linear map
\[
\begin{tikzcd}
U\ten_{l^e} \Hom_{l^e}(l,l^e) \arrow{r}{\sim} & \Hom_{l^e}(l,U) \arrow{r}{\sim} & U^l
\end{tikzcd}
\]
is bijective. By Lemma~\ref{lemma:l-dual} applied to $l^e=l\ten_k l^{op}$ with the trace
$\tr\ten \tr$ and $M=l$, we have the isomorphism
\[
\Hom_k(l,k) \simeq \Hom_{l^e}(l, l^e)
\]
mapping a linear form $\phi$ on $l$ to the morphism mapping $a$ to 
$\sum_{i,j} e_i \ten e_j \, \phi(f_j a f_i)$,
where $(e_i)$ and $(f_i)$ are dual bases of $l$ for the given trace $\tr$.
One easily checks that the image of $\phi=\tr$ is the morphism mapping $a$ to
$\si a = a \si$ and in particular $1_l$ to $\si$. This 
implies the statement.
\end{proof}

We denote the inverse of the bijection in Proposition~\ref{prop:invariant space} by $m \mapsto m^\dag$.

\subsection{On pseudocompact dg algebras} \label{ss:on pseudo-compact dg algebras}

Let $A$ be a pseudocompact dg $k$-algebra. We define $\cc(A)$ to be the category of pseudocompact dg $A$-modules and consider it as enriched over the category of $k$-vector spaces (not pseudocompact $k$-vector spaces).
We write $\cd(A)$ for the (unbounded) derived category of $A$ in the sense of section~5 of
\cite{VandenBergh15}. So the objects in $\cd(A)$ are the pseudocompact dg $A$-modules
and its morphisms are obtained from morphisms of pseudocompact dg $A$-modules by
localising with respect to a suitable class of weak equivalences (which is usually strictly contained in the class of quasi-isomorphisms). We consider $\cd(A)$ as enriched over the category of $k$-vector spaces (not pseudocompact $k$-vector spaces). 
Its thick subcategory generated by the free dg $A$-module of rank one is the {\em perfect derived category} $\per A$. Note that usually, it does not consist of compact objects in $\cd(A)$.
We define the {\em perfectly valued derived category} $\pvd A$ to be the full subcategory of the 
{\em perfectly valued dg modules} in $\cd(A)$, \ie those dg modules $M$ whose homology is of finite total dimension. Clearly, an object $M$ of $\cd(A)$ belongs to $\pvd A$ if and only if the object $\RHom_A(A,M)$ belongs to $\per k=\pvd k$.
For an algebraic triangulated category $\cc$, we write $\cc_{dg}$ for its canonical dg enhancement.
Note that one can also consider the above categories as enriched over the category of pseudocompact $k$-vector spaces. We denote the pseudocompact morphism space between objects $M$ and $N$ 
in $\cc(A)$ by $\Hom_A^{\pc}(M,N)$. It is the pseudocompact $k$-vector space obtained as the limit of the $\Hom_A(M,N)/V$, where $V$ runs through subspaces of finite codimension containing $\Hom_A(M,N')$ for some open dg $A$-submodule $N'$ of $N$. By deriving the pseudocompact $\Hom$-functor we obtain a pseudocompact enrichment for $\cd(A)$ and the dg enhancement $\cd_{dg}(A)$ of $\cd(A)$.

Recall that $A$ is {\em connective} if its homology $H^p(A)$ vanishes for all $p>0$. 
In this case, the derived category $\cd(A)$ has a canonical t-structure whose aisles are
\begin{align*}
\cd(A)^{\leq 0} &= \{M\in \cd(A)\mid H^p(M)=0\mbox{ for all }p>0\}\mbox{ and}\\
\cd(A)^{\geq 0} &= \{M\in \cd(A)\mid H^p(M)=0\mbox{ for all }p<0\}\: .
\end{align*}
Its heart is equivalent to the module category of $H^0(A)$. The dg algebra $A$ is a {\em stalk algebra} if its homology $H^p(A)$ vanishes for all $p\ne 0$. In this case, we have the quasi-isomorphisms
\[
\begin{tikzcd}
A  & \tau_{\leq 0}(A) \arrow[swap]{l}{\sim} \arrow{r}{\sim} & H^0(A)
\end{tikzcd}
\]
so that $A$ is quasi-isomorphic to the ordinary algebra $H^0(A)$.
 
We write $A^e$ for the {\em dg enveloping algebra} $A\ten_k A\op$. Recall that $A$ is {\em smooth} if 
$A$ is perfect over $A^e$ and that $A$ is {\em proper} if its underlying complex is perfect over $k$. The dg algebra $A$ is proper if and only if its homology $H^p(A)$ is finite-dimensional for all integers $p$ and vanishes for all $|p|\gg 0$. The following proposition is proved in Lemma~4.1 of \cite{Keller08d}. We include our own proof for the convenience of the reader. Recall that a $k$-linear category is $\Hom$-finite if all morphism spaces between its objects are finite-dimensional over $k$.

\begin{proposition} \label{prop:finite dimension}
Suppose that $A$ is smooth.
\begin{itemize}
\item[a)] The subcategory $\pvd A$ is contained in $\per A$.
\item[b)] The subcategory $\pvd A$ is $\Hom$-finite.
\end{itemize}
\end{proposition}
\begin{proof}
a) Let $M$ be an object in $\pvd A$. Then the underlying complex of $M$ lies in $\per k$ and therefore, the object $M\lten_k A$ lies
in $\per A$. Since $A$ is perfect over $A^e$, the object $M \liso M \lten_A A$ lies in the thick subcategory of $\cd(A)$ generated by $M\lten_A A^e \iso M\lten_k A$ and so in $\per A$.

b) For any objects $P\in \per A$ and $M\in \pvd A$, the complex $\RHom_A(P,M)$ belongs to $\per k$. Therefore, the statement follows from part~a).
\end{proof}

\subsection{Augmented morphisms between pseudocompact dg algebras} \label{ss:morphisms between pc algebras}

Suppose that $l$ is a finite-dimensional semisimple $k$-algebra. Following section~5 of \cite{VandenBergh15}, an {\em $l$-augmented dg algebra} is a dg $k$-algebra $A$ endowed with dg $k$-algebra morphisms
\[
\begin{tikzcd}
l \arrow{r}{\eta} & A \arrow{r}{\eps} & l
\end{tikzcd}
\]
satisfying $\eps \circ \eta =\id_{l}$. We call $\eta$ the unit map and $\eps$ the augmentation map.  A {\em morphism of $l$-augmented dg algebras} is a morphism $\alpha \colon A\to A'$ of dg $k$-algebras which is compatible with the unit and the augmentation maps.

Before going further, let us recall that unless we specify otherwise, we do {\em not} assume that morphisms between algebras preserve units. For example, the morphism $k \to k \times k$ mapping $a$ to $(a, 0)$ and the morphism $k \to \mathrm{M}_2(k)$ mapping $a$ to the matrix
$\begin{bmatrix}
a & 0 \\
0 & 0
\end{bmatrix}$
are morphisms between $k$-algebras in our sense.

Suppose that $l_A$ and $l_B$ are finite-dimensional semisimple $k$-algebras and $\varphi \colon l_B \to l_A$ is a morphism between $k$-algebras (not necessarily preserving the unit!). A {\em $\varphi$-augmented morphism between dg algebras} is a morphism $f\colon B \to A$ between dg $k$-algebras (not necessarily preserving the unit!) which fits into the commutative diagram
\[
\begin{tikzcd}
l_B \arrow{r}{\eta_B} \arrow[swap]{d}{\varphi} & B \arrow{r}{\eps_B} \arrow{d}{f} & l_B \arrow{d}{\varphi} \\
l_A \arrow[swap]{r}{\eta_A} & A \arrow[swap]{r}{\eps_A} & l_A\mathrlap{\: .}
\end{tikzcd}
\]
A {\em morphism of $\varphi$-augmented morphisms between dg algebras} is given by morphisms
\[
\alpha \colon A\to A'\quad \mbox{and}\quad \beta \colon B\to B'
\]
of dg $k$-algebras which are compatible with the unit and the augmentation maps and fit into the commutative square
\[
\begin{tikzcd}
B\arrow[swap]{d}{f}\arrow{r}{\beta} & B'\arrow{d}{f'} \\
A\arrow[swap]{r}{\alpha} & A'\mathrlap{\: .}
\end{tikzcd}
\]
If the above dg algebras are pseudocompact and the morphisms are continuous, we obtain the definition of {\em augmented pseudocompact dg algebras} and {\em augmented morphisms between pseudocompact dg algebras}.
For a pseudocompact dg algebra $A$, its {\em radical} $\rad A$ is defined to be the common annihilator of all the simple pseudocompact dg $A$-modules. An $l$-augmented pseudocompact dg algebra $A$ is {\em complete} if the kernel of $\eps$ equals the radical $\rad A$. Denote by $\PCAlgc l$ the category of complete $l$-augmented pseudocompact dg algebras. By section~12.4 of \cite{VandenBergh15}, \cf also section~1.3.1 of \cite{Lefevre03}, it has a cofibrantly generated model category structure. Thus, it yields a model category structure on the category $\PCAlgc \varphi$ of $\varphi$-augmented morphisms between complete pseudocompact dg algebras, whose weak equivalences are the componentwise weak equivalences. Similarly, we can define the notion of {\em augmented cospans for complete pseudocompact dg algebras}. These also carry a natural model category structure. By sections~12.1 and 12.4 of \cite{VandenBergh15}, a weak equivalence in $\PCAlgc l$ is a quasi-isomorphism. By part~e) of Proposition~1.3.5.1 of \cite{Lefevre03}, \cf also Proposition~12.2 of \cite{VandenBergh15}, the converse statement is true if the source and target of the morphism are concentrated in non-positive degrees.

\begin{remark}
For a finite-dimensional $k$-algebra $l$ and a pseudocompact dg $l$-bimodule $V$, the completed dg tensor algebra
\[
(T_l V=\prod_{p\geq 0}V^{\ten_l p}, d)
\]
endowed with the product topology is the complete $l$-augmented pseudocompact dg algebra characterised by the universal property: for any complete $l$-augmented pseudocompact dg algebra $A$, any morphism $V\to \ker \eps$, where $\eps \colon A \to l$ is the augmentation map, in the category of pseudocompact dg $l$-bimodules extends uniquely to a morphism $(T_l V, d)\to A$ in the category of complete $l$-augmented pseudocompact dg algebras.
\end{remark}

\begin{proposition} \label{prop:decomposition}
Let $l_A$ and $l_B$ be finite-dimensional semisimple $k$-algebras and let $\varphi \colon l_B \to l_A$ be a morphism between $k$-algebras (which does not necessarily preserve the unit!). 
Denote by $\varphi_*(l_A)\liso l_A \cdot \varphi(1_{l_B})$ the restriction to $l_B$ of the free $l_A$-module of rank one. Then $\varphi$ induces the isomorphism $l_B \iso \varphi_*(l_A)$ of $l_B$-modules 
if and only if $\varphi$ is a section of $k$-algebras.
\end{proposition}
\begin{proof}
The sufficiency of the condition is obvious. Let us prove that it is necessary.
Since the finite-dimensional $k$-algebras $l_A$ and $l_B$ are semisimple, by the Wedderburn--Artin theorem, we may and will assume that they are products of matrix algebras over division algebras over $k$. Explicitly, we have
\[
l_A=\prod^p_{i=1}\mathrm{M}_{n_i}(D_i)\quad \mbox{and}\quad l_B=\prod^q_{j=1}\mathrm{M}_{m_j}(E_j)\: .
\]
Thus, we have the equivalences
\[
\mod l_A\simeq \bigoplus^p_{i=1}\mod D_i\quad \mbox{and}\quad \mod l_B\simeq\bigoplus^q_{j=1}\mod E_j
\]
of the corresponding finite-dimensional module categories. The composed functor of the restriction $\varphi_*$ with the induction $\varphi^*$ maps the $l_B$-module $l_B$ to $l_A\cdot \varphi(1_{l_B})$, which is isomorphic to $l_B$ via $\varphi$ by the assumption. Since the $k$-algebra $l_B$ is semisimple, this implies that the unit of the adjunction $(\varphi^*,\varphi_*)$ is a natural isomorphism. Therefore, the induction functor $\varphi^*$ is fully faithful. So it maps the indecomposable object $E_j$ to some $D_{i_j}$. Moreover, if $j_1$ and $j_2$ are distinct, the images of $E_{j_1}$ and $E_{j_2}$ cannot lie in the same block. Therefore, the induction functor $\varphi^*$ factors into the direct sum of fully faithful functors $\mod E_j\to \mod D_{i_j}$ followed by the canonical embedding
\[
\bigoplus^q_{j=1}\mod D_{i_j}\longrightarrow \bigoplus^p_{i=1}\mod D_i\: .
\]
This implies that the image of $\mathrm{M}_{m_j}(E_j)$ under $\varphi$ is contained in $\mathrm{M}_{n_{i_j}}(D_{i_j})$. By full faithfulness again, the $k$-algebra $D_{i_j}$ is isomorphic to $E_j$. The morphism $\varphi$ is injective because for any $b$ in $l_B$ such that $\varphi(b)=0$, the image of $b$ under the $l_B$-module isomorphism $l_B\iso \varphi_*(l_A)$ is $0$. We deduce that we have $m_j\leq n_{i_j}$. If we compare the dimensions over $E_j$ on both sides of the isomorphism 
$\mathrm{M}_{m_j}(E_j)\iso \varphi_*(\mathrm{M}_{n_{i_j}}(E_{j}))$, we see that we
have $m^2_j\geq n_{i_j}m_j$, so $m_j\geq n_{i_j}$. We conclude that $m_j$ equals $n_{i_j}$ and that the morphism $\varphi$ is a bijection onto $\prod^q_{j=1}\mathrm{M}_{n_{i_j}}(D_{i_j})$. This implies the statement.
\end{proof}

\subsection{Calabi--Yau structures} \label{ss:Calabi-Yau structures}

In this section, we recall the necessary background on Hochschild and cyclic homology, absolute and relative Calabi--Yau structures and Calabi--Yau cospans. We work in the setting of (pseudocompact) dg algebras but everything generalises to the setting of small dg categories (enriched over the category of pseudocompact vector spaces).

Following section~1 of \cite{Kassel87}, a {\em mixed complex} over $k$ is a dg module over the dg algebra $\La=k[t]/(t^2)$, where $t$ is an indeterminate of degree $-1$ satisfying $d(t)=0$. Let $l$ be a finite-dimensional separable $k$-algebra. For a dg $l$-algebra $A$, its {\em mixed complex} $\mathrm{M}(A)$ is defined as follows. Its underlying complex is defined to be the cone of the map $\id-\tau$ from the sum total complex $\rb^+(A)_l$ of
\[
\begin{tikzcd}
  \cdots\arrow{r} & (A^{\ten_l 3})_l \arrow{r}{b'} & (A^{\ten_l 2})_l \arrow{r}{b'} & A_l
\end{tikzcd}
\]
to the sum total complex $\rc(A)$ of
\[
\begin{tikzcd}
  \cdots\arrow{r} & (A^{\ten_l 3})_l \arrow{r}{b} & (A^{\ten_l 2})_l \arrow{r}{b} & A_l \: .
\end{tikzcd}
\]
Here $\tau$ maps $a_1\ten\cdots\ten a_p$ to
\[
(-1)^{(|a_p|+1)(p-1+|a_1|+\cdots+|a_{p-1}|)}a_p\ten a_1\ten \cdots\ten a_{p-1}\: ,
\]
the differential of $(A^{\ten_l p})_l$ maps $a_1\ten\cdots\ten a_p$ to
\[
\sum_{i=1}^p (-1)^{i-1+|a_1|+\cdots+|a_{i-1}|}a_1\ten\cdots\ten d(a_i)\ten\cdots\ten a_p \: ,
\]
the map $b$ is the differential of the Hochschild chain complex and $b'$ is induced by that of the augmented bar resolution. Explicitly, the differential $b$ maps $a_1\ten\cdots\ten a_p$ to
\begin{align*}
& \sum_{i=1}^{p-1}(-1)^{i-1+|a_1|+\cdots+|a_i|} a_1\ten\cdots\ten a_i a_{i+1}\ten\cdots\ten a_p \\
& +(-1)^{(|a_p|+1)(p+|a_1|+\cdots+|a_{p-1}|)-1}a_p a_1\ten\cdots\ten a_{p-1}
\end{align*}
and $b'$ maps $a_1\ten\cdots\ten a_p$ to 
\[
\sum_{i=1}^{p-1}(-1)^{i-1+|a_1|+\cdots+|a_i|}a_1\ten\cdots\ten a_i a_{i+1}\ten\cdots\ten a_p\: .
\]
The $\La$-module structure on $\mathrm{M}(A)$ is determined by the action of $t$, which vanishes on $\rb^+(A)_l$ and maps the component $(A^{\ten_l p})_l$ of $\rc(A)$ to the corresponding component of $\rb^+(A)_l$ via the map $\sum_{i=0}^{p-1}\tau^i$.

The {\em Hochschild complex} $\HH(A)$ of $A$ is defined to be the underlying complex of $\mathrm{M}(A)$. By construction, we have the canonical triangle
\[
\begin{tikzcd}
\rb^+(A)_l \arrow{r}{\id-\tau} & \rc(A) \arrow{r} & \mathrm{M}(A) \arrow{r} & \Si \rb^+(A)_l
\end{tikzcd}
\]
in $\cd(k)$. The complex $\rb^+(A)_l$ is contractible (since it is the sum total complex of a contractible complex of complexes) so that the morphism $\rc(A) \to \mathrm{M}(A)$ is a quasi-isomorphism. This shows that our definition of the Hochschild complex coincides with the classical one up to a canonical quasi-isomorphism.

Let $\bp k$ be the minimal cofibrant resolution of $k$ as a dg $\La$-module. The {\em cyclic complex} $\HC(A)$ of $A$ is defined to be the complex $\mathrm{M}(A)\ten_\La \bp k$. The {\em negative cyclic complex} $\HN(A)$ of $A$ is defined to be the complex $\Hom_{\cc_{dg}(\La)}(\bp k, \mathrm{M}(A))$. The {\em periodic cyclic complex} $\HP(A)$ of $A$ is defined to be the inverse limit of the system
\[
\cdots \longrightarrow \Hom_{\cc_{dg}(\La)}(\Si^4 \bp k, \mathrm{M}(A)) \longrightarrow \Hom_{\cc_{dg}(\La)}(\Si^2 \bp k, \mathrm{M}(A)) \longrightarrow \Hom_{\cc_{dg}(\La)}(\bp k, \mathrm{M}(A))\: ,
\]
where the transition maps are induced by the canonical map $\bp k\to \Si^2 \bp k$. Their homologies are called {\em Hochschild homology} $\HH_*(A)$, {\em cyclic homology} $\HC_*(A)$, {\em negative cyclic homology} $\HN_*(A)$, {\em periodic cyclic homology} $\HP_*(A)$, respectively. The ISB triangle
\[
\begin{tikzcd}
  \HH(A)\arrow{r}{I} & \HC(A)\arrow{r}{S} & \Si^2 \HC(A)\arrow{r}{B} & \Si \HH(A)
\end{tikzcd}
\]
in the homotopy category of complexes yields Connes' long exact sequence
\[
\begin{tikzcd}
\cdots \arrow{r} & \HH_{d+1}(A)\arrow{r}{I} & \HC_{d+1}(A)\arrow{r}{S} & \HC_{d-1}(A)\arrow{r}{B} & \HH_d(A)\arrow{r} & \cdots
\end{tikzcd}
\]
which relates Hochschild and cyclic homology, \cf the complex~(7.2) of \cite{VandenBergh15}. Here Connes' map $B$ factors through the canonical map $\HN_d(A)\to \HH_d(A)$. The {\em reduced} version of each type of complex and homology is obtained by applying the above constructions to the quotient mixed complex $\mathrm{M}(A)/\mathrm{M}(l)$. We have the reduced version of the ISB triangle as well. We refer the reader to section~6.1 of \cite{VandenBergh15} for the analogues of the above constructions for pseudocompact dg algebras (where sum total complexes have to be replaced with product total complexes). In the sequel, we use them freely. The reduced cyclic homology of completed dg tensor algebras can be calculated by the following proposition.

\begin{proposition} \label{prop:cyclic homology}
If the field $k$ is of characteristic $0$, then the reduced cyclic homology $\HC^{\red}_*(A)$ of the pseudocompact dg algebra $A=(T_l V,d)$ is isomorphic to 
\[
H^{-*}((A/(l+[A,A]))_l)\: .
\]
\end{proposition}
\begin{proof}
By the variant of Proposition~3.1.5 of \cite{Loday98} for the reduced cyclic homology of pseudocompact dg algebras, the reduced cyclic homology $\HC^{\red}_*(A)$ is isomorphic to the homology of the product total complex of
\begin{equation} \label{eq:double complex 1}
\begin{tikzcd}
\cdots \arrow{r}{\gamma} & (A\ten_l V)_l \arrow{r}{b} & (A/l)_l \arrow{r}{\gamma} & (A\ten_l V)_l \arrow{r}{b} & (A/l)_l \arrow{r} & 0\: .
\end{tikzcd}
\end{equation}
Since $\gamma$ and $b$ are homogeneous with respect to the tensor degree, 
we can consider the homology of its rows for an arbitrary tensor degree $p$. It is
\begin{equation} \label{eq:complex}
\begin{tikzcd}
\cdots \arrow{r}{\gamma} & (V^{\ten_l p})_l \arrow{r}{b} & (V^{\ten_l p})_l \arrow{r}{\gamma} & (V^{\ten_l p})_l \arrow{r}{b} & (V^{\ten_l p})_l \arrow{r} & 0\: ,
\end{tikzcd}
\end{equation}
where $b$ maps $v_1\ldots v_p$ to 
\[
v_1\ldots v_p-(-1)^{(|v_1|+\cdots +|v_{p-1}|)|v_p|}v_p v_1\ldots v_{p-1}\: ,
\]
and $\gamma$ maps $v_1\ldots v_p$ to 
\[
\sum_{i=1}^p(-1)^{(|v_1|+\cdots +|v_{i}|)(|v_{i+1}|+\cdots +|v_p|)}v_{i+1}\ldots v_p v_1\dots v_i\: .
\]
Let $\tau$ be the cyclic permutation which maps $v_1\ldots v_p$ to $(-1)^{(|v_1|+\cdots +|v_{p-1}|)|v_p|}v_p v_1\ldots v_{p-1}$ and put $N=\sum_{i=0}^{p-1}\tau^i$. Then we have $b=\id-\tau$ and $\gamma=N$. Using the flat resolution
\[
\begin{tikzcd}
\cdots \arrow{r}{N} & k[\Z/p\Z] \arrow{r}{\id-\tau} & k[\Z/p\Z] \arrow{r}{N} & k[\Z/p\Z] \arrow{r}{\id-\tau} & k[\Z/p\Z] \arrow{r} & k \arrow{r} & 0
\end{tikzcd}
\]
of $k$ as a $k[\Z/p\Z]$-module we find that the complex~(\ref{eq:complex}) is just $k\lten_{k[\Z/p\Z]}(V^{\ten_l p})_l$. Since $k$ is of characteristic 0, 
the group algebra $k[\Z/p\Z]$ is semisimple, so that the homology of each row of $k\lten_{k[\Z/p\Z]}(V^{\ten_l p})_l$ vanishes in all negative degrees.
As a result, the product total complex of the double complex~(\ref{eq:double complex 1}) is quasi-isomorphic to 
\[
\begin{tikzcd}
(A/l)_l/\im(b\colon (A\ten_l V)_l \to (A/l)_l)\arrow{r}{\sim} & (A/(l+[A,A]))_l \: .
\end{tikzcd}
\]
Therefore, the reduced cyclic homology $\HC^{\red}_*(A)$ is isomorphic to $H^{-*}((A/(l+[A,A]))_l)$.
\end{proof}

Denote by $\Omega^1_l A$ the kernel of the multiplication map $A\ten_l A\to A$ which maps $a\ten b$ to $ab$. We define the operator $D\colon A\to \Omega^1_l A$ by $D(a)=a\ten 1-1\ten a$. For a dg $A$-bimodule $M$, we denote the quotient $M/[A,M]$ by $M_\natural$. We define the map $\del_0 \colon A_l\to (\Omega^1_l A)_\natural$ which maps $\ol{a}$ to $\ol{Da}$ and the map $\del_1 \colon (\Omega^1_l A)_\natural \to A_l$ which maps $\ol{aDb}$ to $\ol{[a,b]}$. For a completed dg tensor algebra, Connes' map $B$ can be calculated using the following proposition.

\begin{proposition} \label{prop:connecting morphism}
Let $A=(T_l V,d)$ be a pseudocompact dg algebra and $p$ an integer. If the field $k$ is of characteristic $0$, then Connes' map $B\colon \HC^{\red}_{p-1}(A)\to \HH^{\red}_p(A)$ identifies with
\begin{tikzcd}
\begin{bmatrix}
0 \\
-\del_0
\end{bmatrix}
\colon H^{1-p}((A/(l+[A,A]))_l)\arrow{r} & H^{-p}(\cone(\del_1 \colon (\Omega^1_l A)_\natural \to (A/l)_l))\: .
\end{tikzcd}
\end{proposition}
\begin{proof}
By the variant of Proposition~6.2 of \cite{VandenBergh15} for pseudocompact dg algebras, the complex $\HH^{\red}(A)$ is quasi-isomorphic to the product total complex of
\[
\begin{tikzcd}
(\Omega^1_l A)_\natural \arrow{r}{\del_1} & (A/l)_l
\end{tikzcd}
\]
and the complex $\HC^{\red}(A)$ is quasi-isomorphic to the product total complex of
\begin{equation} \label{eq:double complex 2}
\begin{tikzcd}
\cdots\arrow{r}{\del_0} & (\Omega^1_l A)_\natural \arrow{r}{\del_1} & (A/l)_l \arrow{r}{\del_0} & (\Omega^1_l A)_\natural \arrow{r}{\del_1} & (A/l)_l \: .
\end{tikzcd}
\end{equation}
If we use these complexes, the reduced ISB triangle
\[
\begin{tikzcd}
  \HH^{\red}(A)\arrow{r}{I} & \HC^{\red}(A)\arrow{r}{S} & \Si^2 \HC^{\red}(A)\arrow{r}{B} & \Si \HH^{\red}(A)
\end{tikzcd}
\]
is induced by a graded split sequence of complexes. Thus, Connes' map $B$ is given by
$-\del_0$. By Proposition~\ref{prop:cyclic homology}, the product total complex of the double complex ~(\ref{eq:double complex 2}) is quasi-isomorphic to $\cok \del_1 \iso (A/(l+[A,A]))_l$. Now the statement follows since the map $\del_0$ factors through $\cok \del_1$.
\end{proof}

We now recall the notions of absolute Calabi--Yau structures from section~3 of the original article \cite{BravDyckerhoff19}, section~4 of the article \cite{Yeung16} or section~10.3 of the survey article \cite{KellerWang23}. Fix an integer $d$. Let $A$ be a smooth pseudocompact dg $l$-algebra. A {\em (absolute left) $d$-Calabi--Yau structure} on $A$ is a class $[\tilde{\xi}]$ in $\HN_d(A)$ whose image $[\xi]$ under the canonical map $\HN_d(A)\to \HH_d(A)$ is {\em non-degenerate}, \ie the morphism $\Si^d A^\vee \to A$ in $\cd(A^e)$ obtained from $[\xi]$ via
\[
\begin{tikzcd}
\HH_d(A)\arrow{r}{\sim} & H^{-d}(\RHom_{A^e}^{\pc}(A^\vee,A))\arrow[no head]{r}{\sim} & \Hom_{\cd(A^e)}^{\pc}(\Si^d A^\vee,A)
\end{tikzcd}
\]
is an isomorphism, where we denote the derived bimodule dual $\RHom_{A^e}^{\pc}(A,A^e)$ of $A$ by $A^\vee$. An {\em exact (absolute left) $d$-Calabi--Yau structure} on $A$ is a class in $\HC_{d-1}(A)$ such that its image under the canonical map $\HC_{d-1}(A)\to \HN_d(A)$ is a $d$-Calabi--Yau structure on $A$.

Let $A$ be a proper dg $l$-algebra (not supposed to be pseudocompact). A {\em (absolute) right $d$-Calabi--Yau structure} on $A$ is a class $[\tilde{x}]$ in $D\HC_{-d}(A)$ whose image $[x]$ under the canonical map $D\HC_{-d}(A)\to D\HH_{-d}(A)$ is {\em non-degenerate}, \ie the morphism $A\to \Si^{-d}DA$ in $\cd(A^e)$ obtained from $[x]$ via 
\[
\begin{tikzcd}
D\HH_{-d}(A)\arrow[no head]{r}{\sim} & H^{-d}(\RHom_{A^e}(A, DA))\arrow[no head]{r}{\sim} & \Hom_{\cd(A^e)}(A,\Si^{-d}DA)
\end{tikzcd}
\]
is an isomorphism.

Recall that a $k$-linear $\Hom$-finite triangulated category $\cc$ is {\em $d$-Calabi--Yau} if it is endowed with bifunctorial bijections
\[
\Hom_{\cc}(X, Y) \xlongrightarrow{_\sim} D\Hom_{\cc}(Y, \Si^d X) \ko X, Y \in \cc \: .
\]

\begin{lemma} \label{lem:symplectic}
Suppose that $\cc$ is a $k$-linear $d$-Calabi--Yau triangulated category. Let $X$ be an object in $\cc$. Then the pseudocompact graded $k$-vector space $\Si^{-1}\D\Hom^*_{\cc}(X, X)$ carries a canonical symplectic form.
\end{lemma}
\begin{proof}
By Proposition~A.5.2 of \cite{Bocklandt08}, for any object $X$ in $\cc$, the $d$-Calabi--Yau structure on $\cc$ yields a canonical non-degenerate and graded symmetric bilinear form 
$\langle?, -\rangle$ of degree $-d$ on the graded vector space $\Hom^*_{\cc}(X, X)$. We define the bilinear form $\langle?, -\rangle'$ on the graded vector space $\Si \Hom^*_{\cc}(X, X)$ to be the composition
\[
\langle?, -\rangle'=\langle?, -\rangle \circ(s^{-1}\ten s^{-1})\: .
\]
Then for any $f\in \Hom^i_{\cc}(X, X)$ and $g\in \Hom^{d-i}_{\cc}(X, X)$, we have
\begin{align*}
\langle sf, sg\rangle' &=\langle(s^{-1}\ten s^{-1})(sf, sg)\rangle \\
                       &=(-1)^{i-1}\langle f, g\rangle \\
                       &=(-1)^{i-1+i(d-i)}\langle g, f\rangle \\
                       &=(-1)^{d+i(d-i)}\langle(s^{-1}\ten s^{-1})(sg, sf)\rangle \\
                       &=-(-1)^{(i-1)(d-i-1)}\langle sg, sf\rangle'\: .
\end{align*}
This implies that the bilinear form $\langle?, -\rangle'$ on the graded vector space $\Si \Hom^*_{\cc}(X, X)$ is non-degenerate and graded anti-symmetric. Therefore, its graded dual
\[
\D \Si \Hom^*_{\cc}(X, X)=\Si^{-1}\D\Hom^*_{\cc}(X, X)
\]
also carries a canonical symplectic form.
\end{proof}

\begin{corollary} \label{cor:symplectic}
Suppose that $A$ is a (smooth) pseudocompact dg $l$-algebra which carries a $d$-Calabi--Yau structure. Let $X$ be an object in $\pvd A$. Then the pseudocompact graded $k$-vector space $\Si^{-1}\D\Ext^*_A(X, X)$ carries a canonical symplectic form.
\end{corollary}
\begin{proof}
The statement follows from the variant of Lemma~4.1 of \cite{Keller08d} for pseudocompact dg algebras and Lemma~\ref{lem:symplectic} immediately.
\end{proof}

The construction of the mixed complex $\mathrm{M}(A)$ is functorial with respect to (not necessarily unital!) morphisms between dg algebras as defined in section~\ref{ss:morphisms between pc algebras}. For a morphism $f\colon B\to A$ between dg algebras, its {\em relative mixed complex} $\mathrm{M}(A, B)$ is defined to be the cone of the induced morphism from $\mathrm{M}(B)$ to $\mathrm{M}(A)$. We define the {\em relative} versions of the Hochschild complex, the cyclic complex, the negative cyclic complex and the periodic cyclic complex by applying the above constructions to the mixed complex $\mathrm{M}(A, B)$ instead of $\mathrm{M}(A)$. With these definitions, we obtain the relative and the reduced relative versions of the ISB triangle as well. We proceed analogously for a morphism $B \to A$ between pseudocompact dg algebras.

We now succinctly recall the definition of relative Calabi--Yau structures. For more leisurely
accounts, we refer the reader to section~4 of the original article \cite{BravDyckerhoff19}, section~4 of the article \cite{Yeung16} or section~10.12 of the survey article \cite{KellerWang23}. Let $f\colon B\to A$ be a morphism between smooth pseudocompact dg algebras. We write $\mu \colon A\lten_B A \to A$ for the morphism in $\cd(A^e)$ corresponding to the morphism $f\colon B \to A$ in $\cd(B^e)$ under the bijection
\[
\begin{tikzcd}
\Hom_{\cd(A^e)}(A\lten_B A, A) \arrow{r}{\sim} & \Hom_{\cd(B^e)}(B, A) \: .
\end{tikzcd}
\]
A {\em relative (left) $d$-Calabi--Yau structure} on $f$ is a class $[(\tilde{\xi_A},s\tilde{\xi_B})]$ in $\HN_d(A,B)$ whose image $[(\xi_A,s\xi_B)]$ under the canonical map
\[
\HN_d(A,B)\longrightarrow \HH_d(A,B)
\]
is {\em non-degenerate}, \ie the morphism $[\hat{\xi_B}s^{1-d}]\colon \Si^{d-1}B^\vee \to B$ in $\cd(B^e)$ and the morphism
\[
\begin{tikzcd}
  \Si^{d-1}A^\vee\arrow{rr}{(-1)^{d-1}\Si^{d-1}\mu^\vee}\arrow{d}{\Si^{-1}[\hat{\xi}'']} & & \Si^{d-1}(A\lten_B A)^\vee\arrow{r}\arrow{d}{-\mathbf{L}f^{e*}[\hat{\xi_B}s^{1-d}]}& \Si^{d}(\cone \mu)^\vee\arrow{r}\arrow{d}{[\hat{\xi}']} & \Si^d A^\vee\arrow{d}{[\hat{\xi}'']}\\
  \Si^{-1}\cone \mu \arrow{rr} & & A\lten_{B}A\arrow[swap]{r}{\mu} & A\arrow{r} & \cone \mu
\end{tikzcd}
\]
of triangles in $\cd(A^e)$ obtained from $[(\xi_A,s\xi_B)]$ are isomorphisms, where these morphisms are constructed as follows: let us write $X(A)$ for a cofibrant resolution of $A$ as a pseudocompact dg $A$-bimodule and similarly for $B$. These resolutions allow us to describe the given classes using representatives. We denote the graded morphism of degree $-d$ corresponding to the representative $\xi_A\in A\ten_{A^e}X(A)$ by $\hat{\xi_A}\colon X(A)^\vee \to A$ and the (closed) morphism of degree $1-d$ corresponding to $\xi_B\in B\ten_{B^e}X(B)$ by $\hat{\xi_B}\colon X(B)^\vee \to B$. We use $\hat{\xi}'$ and $\hat{\xi}''$ to denote the morphisms
\[
\begin{bmatrix}
(-1)^d \mu\circ f^{e*}(\hat{\xi_B}s^{1-d}) & \hat{\xi_A}s^{-d}
\end{bmatrix}
\colon \Si^{d}(\cone \mu)^\vee \longrightarrow A
\]
respectively
\[
\begin{bmatrix}
(-1)^d \hat{\xi_A}s^{-d} \\
(-1)^d \Si f^{e*}(\hat{\xi_B}s^{1-d})\circ \Si^d\mu^\vee
\end{bmatrix}
\colon \Si^{d}X(A)^\vee \longrightarrow \cone \mu
\]
in $\cc(A^e)$ and use the identification $(A\lten_B A)^\vee \simeq A\lten_B B^\vee\lten_B A$, which holds since $B$ is perfect over $B^e$. In the notation $f^{e*}(\hat{\xi_B}s^{1-d})$, we use $f^{e*}$ to denote the induction functor $\cc(B^e) \to \cc(A^e)$. Note that the definition implies that the class $[\tilde{\xi_B}]$ in $\HN_{d-1}(B)$ is a \linebreak $(d-1)$-Calabi--Yau structure on $B$. If the dg algebra $B$ vanishes, then we recover the absolute notion. An {\em exact relative (left) $d$-Calabi--Yau structure} on $f$ is a class in $\HC_{d-1}(A,B)$ such that its image under the canonical map $\HC_{d-1}(A,B)\to \HN_d(A,B)$ is a relative $d$-Calabi--Yau structure on $f$.

Let $f\colon A\to B$ be a morphism between proper dg algebras (not supposed to be pseudocompact). A {\em relative right $d$-Calabi--Yau structure} on $f$ is a class $[(s\tilde{x_B},\tilde{x_A})]$ in \linebreak $D\HC_{1-d}(B,A)$ whose image $[(sx_B,x_A)]$ under the canonical map
\[
D\HC_{1-d}(B,A)\longrightarrow D\HH_{1-d}(B,A)
\]
is {\em non-degenerate}, \ie the morphism $[s^{1-d}\hat{x_B}]\colon B\to \Si^{1-d}DB$ in $\cd(B^e)$ and the morphism
\[
\begin{tikzcd}
  \cocone f \arrow{r}\arrow{d}{[\hat{x}'']}&A\arrow{r}{f}\arrow{d}{[\hat{x}']}&B\arrow{rr}\arrow{d}{(-1)^{d-1}\mathbf{R}f^e_*[s^{1-d}\hat{x_B}]}& &\Si\cocone f\arrow{d}{\Si[\hat{x}'']}\\
  \Si^{-d}DA\arrow{r}&\Si^{-d}D\cocone f\arrow{r}&\Si^{1-d}DB\arrow[swap]{rr}{(-1)^{1-d}\Si^{1-d}Df}& &\Si^{1-d}DA
\end{tikzcd}
\]
of triangles in $\cd(A^e)$ obtained from $[(sx_B,x_A)]$ are isomorphisms, where these morphisms are constructed as follows: let us write $X(A)$ for a cofibrant resolution of $A$ as a dg \linebreak $A$-bimodule and similarly for $B$. These resolutions allow us to describe the given classes using representatives. We denote the graded morphism of degree $-d$ corresponding to the representative $x_A\in D(A\ten_{A^e}X(A))$ by $\hat{x_A}\colon X(A) \to DA$ and the (closed) morphism of degree $1-d$ corresponding to $x_B\in D(B\ten_{B^e}X(B))$ by $\hat{x_B}\colon X(B) \to DB$. We use $\hat{x}'$ and $\hat{x}''$ to denote the morphisms
\[
\begin{bmatrix}
f^e_*(s^{1-d}\hat{x_B})\circ f \\
s^{-d}\hat{x_A}
\end{bmatrix}
\colon A \longrightarrow \Si^{-d}D\cocone f
\]
respectively
\[
\begin{bmatrix}
\Si^{-d}Df\circ \Si^{-1}f^e_*(s^{1-d}\hat{x_B}) & (-1)^d s^{-d}\hat{x_A}
\end{bmatrix}
\colon \cocone f \longrightarrow \Si^{-d}DA
\]
in $\cc(A^e)$. In the notation $f^e_*(s^{1-d}\hat{x_B})$, we use $f^e_*$ to denote the restriction functor \linebreak $\cc(B^e) \to \cc(A^e)$. Note that the definition implies that the class $[\tilde{x_B}]$ in $D\HC_{1-d}(B)$ is a right $(d-1)$-Calabi--Yau structure on $B$. If the dg algebra $B$ vanishes, then we recover the absolute notion.

A {\em $d$-Calabi--Yau structure on a cospan} in the sense of section~6 of \cite{BravDyckerhoff19}
\[
\begin{tikzcd}
   & B_1\arrow{d}{f_1} \\
  B_2\arrow{r}{f_2} & A
\end{tikzcd}
\]
for smooth pseudocompact dg algebras is a class $[(\tilde{\xi_A},s\tilde{\xi_{B_1}},s\tilde{\xi_{B_2}})]$ in
\[
H^{-d}(\cone(
\begin{bmatrix}
f_1 & -f_2
\end{bmatrix}
\colon \HN(B_1)\oplus \HN(B_2)\to \HN(A)))
\]
whose underlying Hochschild class $[(\xi_A,s\xi_{B_1},s\xi_{B_2})]$ is {\em non-degenerate}, \ie the morphisms $[\hat{\xi_{B_i}}s^{1-d}]\colon \Si^{d-1}B_i^\vee \to B_i$ in $\cd(B_i^e)$ obtained from $[\xi_{B_i}]$ are isomorphisms, where $i=1$, $2$, and the commutative diagram
\[
\begin{tikzcd}
  \Si^{d-1}A^\vee\arrow{rr}{(-1)^{d-1}\Si^{d-1}\mu^\vee}\arrow[swap]{d}{(-1)^{d-1}\Si^{d-1}\mu^\vee}& & \Si^{d-1}(A\lten_{B_1} A)^\vee\arrow{dr}{-\mathbf{L}f^{e*}[\hat{\xi_{B_1}}s^{1-d}]} & \\
  \Si^{d-1}(A\lten_{B_2} A)^\vee\arrow[swap]{drr}{-\mathbf{L}f^{e*}[\hat{\xi_{B_2}}s^{1-d}]} & & & A\lten_{B_1}A\arrow{d}{\mu} \\
  & & A\lten_{B_2}A\arrow[swap]{r}{\mu} & A
\end{tikzcd}
\]
in $\cd(A^e)$ obtained from $[(\xi_A,s\xi_{B_1},s\xi_{B_2})]$ is homotopy (co)Cartesian. In particular, the class $[\tilde{\xi_{B_i}}]$ in $\HN_{d-1}(B_i)$ is a $(d-1)$-Calabi--Yau structure on $B_i$, where $i=1$, $2$. If the dg algebra $B_2$ vanishes, then we recover the relative notion.

\begin{proposition} \label{prop:CY-cospan}
Suppose that $f_1\colon B_1 \to A$ and $f_2\colon B_2 \to A$ are morphisms between smooth pseudocompact dg algebras satisfying
\[
f_1(1_{B_1})\cdot f_2(1_{B_2})=0=f_2(1_{B_2})\cdot f_1(1_{B_1})\: .
\]
Then the class
$[(\tilde{\xi_A},s\tilde{\xi_{B_1}},s\tilde{\xi_{B_2}})]$ is a $d$-Calabi--Yau structure on the cospan
\[
\begin{tikzcd}
   & B_1\arrow{d}{f_1} \\
  B_2\arrow{r}{f_2} & A
\end{tikzcd}
\]
if and only if the class
$[(\tilde{\xi_{A}},s(\tilde{\xi_{B_1}}-\tilde{\xi_{B_2}}))]$ is a relative $d$-Calabi--Yau structure on the morphism
$\begin{bmatrix}
f_1 & f_2
\end{bmatrix}
\colon B_1\times B_2 \to A$.
\end{proposition}
\begin{proof}
On the one hand, the class
$[(\tilde{\xi_A},s\tilde{\xi_{B_1}},s\tilde{\xi_{B_2}})]$ is a $d$-Calabi--Yau structure on the cospan
\[
\begin{tikzcd}
   & B_1\arrow{d}{f_1} \\
  B_2\arrow{r}{f_2} & A
\end{tikzcd}
\]
if and only if the morphisms $[\hat{\xi_{B_i}}s^{1-d}]\colon \Si^{d-1}B_i^\vee \to B_i$ in $\cd(B_i^e)$, where $i=1$, $2$, and the morphism
\begin{equation} \label{eq:cospan duality}
\begin{bmatrix}
-\mu\circ f^{e*}(\hat{\xi_{B_1}}s^{1-d}) & (-1)^{d-1}\hat{\xi_A}s^{1-d} \\
0 & (-1)^d \Si^d\mu^\vee
\end{bmatrix}
\end{equation}
in $\cd(A^e)$ from the cone of
\[
(-1)^{d-1}\Si^{d-1}\mu^\vee \colon \Si^{d-1}A^\vee \longrightarrow \Si^{d-1}(A\lten_{B_1}A)^\vee
\]
to that of
\[
-\mu\circ \mathbf{L}f^{e*}[\hat{\xi_{B_2}}s^{1-d}]\colon \Si^{d-1}(A\lten_{B_2}A)^\vee \longrightarrow A
\]
are isomorphisms. On the other hand, the class $[(\tilde{\xi_{A}},s(\tilde{\xi_{B_1}}-\tilde{\xi_{B_2}}))]$ is a relative $d$-Calabi--Yau structure on the morphism
$\begin{bmatrix}
f_1 & f_2
\end{bmatrix}
\colon B_1\times B_2 \to A$ if and only if the morphism
\[
[\hat{\xi_{B_1}-\xi_{B_2}}s^{1-d}]\colon \Si^{d-1}(B_1\times B_2)^\vee \longrightarrow B_1\times B_2
\]
in $\cd((B_1\times B_2)^e)$ and the morphism
\begin{equation} \label{eq:relative duality}
\begin{bmatrix}
(-1)^d \mu\circ f^{e*}(\hat{\xi_{B_1}-\xi_{B_2}}s^{1-d}) & \hat{\xi_A}s^{-d}
\end{bmatrix}
\end{equation}
in $\cd(A^e)$ from the cone of
\[
(-1)^{d-1}\Si^{d-1}\mu^\vee \colon \Si^{d-1}A^\vee \longrightarrow \Si^{d-1}(A\lten_{B_1\times B_2}A)^\vee
\]
to $A$ is an isomorphism. The assumption $f_1(1_{B_1})\cdot f_2(1_{B_2})=0=f_2(1_{B_2})\cdot f_1(1_{B_1})$ implies that $A\lten_{B_1\times B_2}A$ is isomorphic to $(A\lten_{B_1}A)\oplus(A\lten_{B_2}A)$ in $\cd(A^e)$. Then the statement follows because the cone of the morphism~(\ref{eq:cospan duality}) and that of the morphism~(\ref{eq:relative duality}) multiplied by $(-1)^{d-1}$ are isomorphic.
\end{proof}

\subsection{The necklace bracket} \label{ss:the necklace bracket}

We recall some notions on noncommutative symplectic geometry from \cite{Kontsevich93, CrawleyBoeveyEtingofGinzburg07}. Suppose that $l$ is a finite-dimensional semisimple $k$-algebra and $A$ is a dg $l$-algebra. Recall that $\Omega^1_l A$ is the kernel of the multiplication map $A\ten_l A\to A$. Following section~9.1 of \cite{VandenBergh15}, the tensor algebra $T_A(\Omega^1_l A)$ is bigraded:
by definition, for $\omega$ in $\Omega^1_l A$, the degree derived from the internal degree (given by the grading on $A$) is denoted by $|\omega|$ and the `form degree' is defined as $|\!|\omega|\!|=1$ (if $\omega\neq 0$).
For two homogeneous elements $\omega$ and $\omega'$ of $T_A(\Omega^1_l A)$, we define the bigraded commutator by
\[
[\omega,\omega']=\omega \omega'-(-1)^{|\!|\omega'|\!||\!|\omega''|\!|+|\omega'||\omega''|}\omega'\omega \: .
\]
Let $\mathrm{DR}_l(A)$ be the bigraded vector space
\[
T_A(\Omega^1_l A)/[T_A(\Omega^1_l A), T_A(\Omega^1_l A)]\: .
\]
The differential on $A$ yields a differential of bidegree $(1,0)$ on $T_A(\Omega^1_l A)$ and on $\mathrm{DR}_l(A)$. This makes $T_A(\Omega^1_l A)$ into a differential bigraded $l$-algebra.

We extend the operator $D\colon A\to \Omega^1_l A$ which maps $a$ to $a\ten 1-1\ten a$ to an $l$-derivation $T_A(\Omega^1_l A)\to T_A(\Omega^1_l A)$ whose square vanishes. Clearly, it is of bidegree $(0,1)$ and commutes with $d$.
In this way, the tensor algebra $T_A(\Omega^1_l A)$ becomes a bidifferential 
bigraded \linebreak $l$-algebra. It is easy to see that $D$ descends to a $k$-linear endomorphism of $\mathrm{DR}_l(A)$ which is of bidegree $(0,1)$ 
and commutes with $d$. In particular, both $T_A(\Omega^1_l A)$ and $\mathrm{DR}_l(A)$ are double complexes.

Recall that a {\em double $l$-derivation} defined on $A$ is an $l^e$-linear
derivation defined on $A$ with values in the $A$-bimodule $A\ten_k A$.
For each double $l$-derivation $\delta$, we denote by $i_\delta$ the {\em contraction} associated with $\delta$, \ie the unique double $l$-derivation
\[
i_\delta \colon T_A(\Omega^1_l A) \longrightarrow T_A(\Omega^1_l A)\ten_k T_A(\Omega^1_l A)
\]
such that, for any $a$ in $A$, we have $i_\delta(a)=0$ and $i_\delta(Da)=\delta(a)$.
For any $\omega$ in $T_A(\Omega^1_l A)$, we let $\iota_\delta(\omega)$ be
the element of $T_A(\Omega^1_l A)$ defined by
\[
\iota_\delta(\omega)=
(-1)^{|i_\delta(\omega)'||i_\delta(\omega)''|}i_\delta(\omega)''i_\delta(\omega)'\: .
\]
Recall that following \cite{VandenBergh15}, by abuse of notation we write $u = u' \ten u''$ instead of \linebreak $u= \sum_i u'_i\ten u''_i$. Thus, the map $\iota_\delta$ is the composition of the graded opposite multiplication of $T_A(\Omega^1_l A)$ with $i_\delta$.

Recall that an element $\omega$ of form degree $2$ in $\mathrm{DR}_l(A)$ is {\em bisymplectic}
if it is closed for $D$ and the morphism
\[
\mathrm{Der}_{l}(A, A\ten_k A)\longrightarrow \Omega^1_l A
\]
of $A$-bimodules which maps $\delta$ to $\iota_\delta(\omega)$ is an isomorphism. Here we denote the graded vector space of $l$-bilinear derivations from $A$ to $A\ten_k A$ by $\mathrm{Der}_{l}(A, A\ten_k A)$.
For example, suppose that $V$ is a graded $l$-bimodule of finite total dimension and 
\[
A=T_l V = \prod_{p\geq 0}V^{\ten_l p}
\]
is the completed graded tensor algebra. Then, if the field $k$ is of characteristic $0$, for a non-degenerate and graded anti-symmetric element $\eta$ of $V\ten_{l^e} V$, we define $\omega_\eta$ to be $\frac{1}{2}(D\eta')(D\eta'')$. By section~9.1 of \cite{VandenBergh15}, the element $\omega_\eta$ is a bisymplectic form on $T_l V$.

Suppose that the element $\omega$ of $\mathrm{DR}_l(A)$ is bisymplectic. Following section~4.2 of \cite{CrawleyBoeveyEtingofGinzburg07}, for an element $a$ of $A$, we denote the corresponding {\em Hamiltonian vector field} by $H_a$, \ie the preimage of $Da$ under the above isomorphism $\mathrm{Der}_{l}(A, A\ten_k A)\to \Omega^1_l A$. Then, for elements $a$ and $b$ of $A$, we define the element $\ldb a, b\rdb_\omega$ of $A\ten_k A$ to be $H_a(b)$ and $\{a, b\}_\omega$ to be the image of $\ldb a, b\rdb_\omega$ under the multiplication. By Proposition~A.3.3 of \cite{VandenBergh08a}, 
the map $\ldb?,-\rdb_{\omega}$ is a double Poisson bracket.

Now let $V$ be a graded $l$-bimodule of finite total dimension. As above, denote the associated completed graded tensor algebra by
\[
A=T_l V = \prod_{p\geq 0}V^{\ten_l p}\: .
\]
We consider it as an $l$-augmented $l$-algebra. Suppose that $d\colon A\to A$ is a continuous $l^e$-linear differential making $A$ into a dg algebra.
Let $\omega \in \mathrm{DR}_l(A)$ be a bisymplectic element. Then we have the bracket
$\ldb?,-\rdb_{\omega}$ on $(A,d)$. Let $V'$ be another graded \linebreak $l$-bimodule of finite total dimension. It is easy to check that $\ldb?,-\rdb_\omega$ extends to a unique double Poisson bracket on $T_l(V\oplus V')$ such that we have $\ldb u,v\rdb_{\omega}=0$ if $u$ or $v$ lies in $V'$. By composing
with the multiplication of $T_l(V\oplus V')$ we obtain the corresponding {\em necklace bracket} $\{?, -\}_\omega$.

\subsection{$A_\infty$-algebras and $A_\infty$-modules} \label{ss:Ainfty-algebras and Ainfty-modules}

Following \cite{Keller01}, an {\em $A_\infty$-algebra} is a graded $k$-vector space $A$ endowed with $k$-linear maps $m_n\colon A^{\ten_k n}\to A$ of degree $2-n$, $n\geq 1$, satisfying
\[
\sum_{r+s+t=n}(-1)^{r+st}m_{r+1+t}\circ (\id^{\ten r}\ten m_s \ten \id^{\ten t})=0
\]
for all positive integers $n$, where $r$ and $t$ run through non-negative integers and $s$ through positive integers. It is {\em minimal} if $m_1$ vanishes. If $m_n$ vanishes for all $n>2$, then $d=m_1$ and $m_2$ make $A$ into a dg algebra.
A {\em morphism of $A_\infty$-algebras} $f\colon A\to B$ is a family of $k$-linear maps $f_n \colon A^{\ten_k n}\to B$ of degree $1-n$, $n\geq 1$, satisfying
\[
\sum_{r+s+t=n}(-1)^{r+st}f_{r+1+t}\circ (\id^{\ten r}\ten m_s \ten \id^{\ten t})=
\sum_{i_1+\cdots+i_r=n}(-1)^{\sum_{j=1}^{r-1}(r-j)(i_j-1)} m_r\circ (f_{i_1}\ten \cdots \ten f_{i_r})
\]
for all positive integers $n$, where on the left hand side $r$ and $t$ run through non-negative integers and $s$ through positive integers, on the right hand side $r$ and $i_j$ run through positive integers. For an $A_\infty$-algebra morphism $f$, it is a {\em quasi-isomorphism} if $f_1$ is a quasi-isomorphism, it is {\em strict} if $f_n$ vanishes for all $n>1$. For any $A_\infty$-algebra $A$, we have the minimal $A_\infty$-algebra structure on its homology $H^*(A)$ and a quasi-isomorphism  $H^*(A)\to A$ of $A_\infty$-algebras, \cf Theorem~1 of \cite{Kadeishvili80}. The composition $f\circ g$ of $A_\infty$-algebra morphisms is given by
\[
(f\circ g)_n=\sum_{i_1+\cdots+i_r=n}(-1)^{\sum_{j=1}^{r-1}(r-j)(i_j-1)}f_r\circ (g_{i_1}\ten \cdots \ten g_{i_r})\ko n\geq 1\: ,
\]
where $r$ and $i_j$ run through positive integers.

Let $A$ be an $A_\infty$-algebra. An {\em $A_\infty$-module} over $A$ is a graded $k$-vector space $M$ endowed with $k$-linear maps $m^M_n \colon M\ten_k A^{\ten_k (n-1)}\to M$ of degree $2-n$, $n\geq 1$, satisfying
\[
\sum_{r+s+t=n}(-1)^{r+st}m_{r+1+t}\circ (\id^{\ten r}\ten m_s \ten \id^{\ten t})=0
\]
for all positive integers $n$, where $r$ and $t$ run through non-negative integers and $s$ through positive integers. It is {\em minimal} if $m^M_1$ vanishes. A {\em morphism of $A_\infty$-modules} $f\colon L\to M$ is a family of $k$-linear maps $f_n \colon L\ten_k A^{\ten_k (n-1)}\to M$ of degree $1-n$, $n\geq 1$, satisfying
\[
\sum_{r+s+t=n}(-1)^{r+st}f_{r+1+t}\circ (\id^{\ten r}\ten m_s \ten \id^{\ten t})=
\sum_{r+s=n}(-1)^{(r+1)s} m_{1+s}\circ (f_r\ten \id^{\ten s})
\]
for all positive integers $n$, where on the left hand side $r$ and $t$ run through non-negative integers and $s$ through positive integers, on the right hand side $r$ runs through positive integers and $s$ through non-negative integers. Here we write $m_n$ for both $m_n$ and $m^M_n$. For an $A_\infty$-module morphism $f$, it is a {\em quasi-isomorphism} if $f_1$ is a quasi-isomorphism, it is {\em strict} if $f_n$ vanishes for all $n>1$. For any $A_\infty$-module $M$ over $A$, its homology $H^*(M)$ can be endowed with a structure of minimal $A_\infty$-module over $A$ 
such that there is a quasi-isomorphism $H^*(M)\to M$ of $A_\infty$-modules over $A$, \cf Theorem~2 of \cite{Kadeishvili80}. The composition $f\circ g$ of $A_\infty$-module morphisms is given by
\[
(f\circ g)_n=\sum_{r+s=n}(-1)^{(r-1)s}f_{1+s}\circ (g_r \ten \id^{\ten s})\ko n\geq 1\: ,
\]
where $r$ runs through positive integers and $s$ through non-negative integers.

\section{A Darboux theorem for relative Calabi--Yau structures}

\subsection{Ice quivers with potential} \label{ss:ice quivers with potential}

Let $k$ be a field and $Q=(Q_0,Q_1,s,t)$ a finite quiver. We write $e_i$ for the lazy
path at a vertex $i$ of $Q$.
Let $\tilde{Q}$ be the quiver obtained from $Q$ by adding an arrow $\alpha^*\colon j \to i$ for each arrow $\alpha\colon i \to j$. The {\em completed preprojective algebra} associated with the quiver $Q$ is the quotient of the completed path algebra $k\tilde{Q}$ by the closure of the ideal generated by
\[
\sum_\alpha e_i(\alpha \alpha^*-\alpha^* \alpha)e_i \ko i\in Q_0 \: ,
\]
where $\alpha$ runs through arrows of $Q$.
Two cycles of the same length of $Q$ are {\em cyclically equivalent} if they only differ by a rotation. A {\em potential} $W$ on $Q$ is a formal (possibly infinite) $k$-linear combination of cyclic equivalence classes of cycles which are of length at least $3$. For each arrow $\alpha$ of $Q$, we have the {\em cyclic derivative} $\del_{\alpha}\colon kQ/[kQ, kQ]\to kQ$ which maps the class of a cycle $p$ to
\[
\sum_{\{(u,v)\mid p=u\alpha v\}}vu\: ,
\]
where $u$ and $v$ run through paths in $Q$.
The {\em Jacobian algebra} associated with the quiver with potential $(Q,W)$ is the quotient of the completed path algebra $kQ$ by the closure of the ideal generated by $\del_{\alpha}W$, $\alpha \in Q_1$.
We consider $Q$ as a graded quiver concentrated in degree $0$. If $d$ equals $1$ and $Q_1$ is empty, or $d$ equals $2$ and $W$ vanishes, or $d$ equals $3$, let $\bar{Q}$ be the graded quiver obtained from $Q$ by adding an arrow $\alpha^*\colon j \to i$ of degree $2-d$ for each arrow $\alpha\colon i \to j$ and a loop $t_i$ of degree $1-d$ at each vertex $i$. The {\em $d$-dimensional Ginzburg dg algebra} associated with the quiver with potential $(Q,W)$ is the completed dg path algebra $k\bar{Q}$ with the differential determined by
\[
d(\alpha^*)=\del_{\alpha}W \quad \mbox{and}\quad d(t_i)=\sum_\alpha e_i(\alpha \alpha^*-\alpha^* \alpha)e_i\: ,
\]
where $\alpha$ runs through arrows of $Q$. Note that the homology of degree $0$ of the $2$-dimensional Ginzburg dg algebra is the completed preprojective algebra and the homology of degree $0$ of the $3$-dimensional Ginzburg dg algebra is the Jacobian algebra.

An {\em ice quiver} $(Q,F)$ is a quiver $Q$ with a frozen subquiver $F$ (which is not necessarily full). The vertices, respectively the arrows, in $F$ are called {\em frozen vertices}, respectively {\em frozen arrows}, and the vertices, respectively the arrows, not in $F$ are called {\em non-frozen vertices}, respectively {\em non-frozen arrows}.
For a finite ice quiver $(Q,F)$, let $\tilde{Q}_F$ be the quiver obtained from $Q$ by adding an arrow $\alpha^*\colon j \to i$ for each non-frozen arrow $\alpha\colon i \to j$. The {\em completed relative preprojective algebra} associated with the ice quiver $(Q,F)$ is the quotient of the completed path algebra $k\tilde{Q}_F$ by the closure of the ideal generated by
\[
\sum_\alpha e_i(\alpha \alpha^*-\alpha^* \alpha)e_i \ko i\in Q_0 \setminus F_0 \: ,
\]
where $\alpha$ runs through non-frozen arrows of $Q$.
A {\em potential} $W$ on $(Q,F)$ is a potential on $Q$. The {\em relative Jacobian algebra} associated with the ice quiver with potential $(Q,F,W)$ is the quotient of the completed path algebra $kQ$ by the closure of the ideal generated by $\del_{\alpha}W$, $\alpha \in Q_1 \setminus F_1$.
We consider $Q$ as a graded quiver concentrated in degree $0$. If $d$ equals $2$ and $W$ vanishes or $d$ equals $3$, let $\bar{Q}_F$ be the graded quiver obtained from $Q$ by adding an arrow $\alpha^*\colon j \to i$ of degree $2-d$ for each non-frozen arrow $\alpha\colon i \to j$ and a loop $t_i$ of degree $1-d$ at each non-frozen vertex $i$. The {\em $d$-dimensional relative Ginzburg dg algebra} associated with the ice quiver with potential $(Q,F,W)$ is the completed dg path algebra $k\bar{Q}_F$ with the differential determined by
\[
d(\alpha^*)=\del_{\alpha}W\quad \mbox{and}\quad d(t_i)=\sum_\alpha e_i(\alpha \alpha^*-\alpha^* \alpha)e_i \: ,
\]
where $\alpha$ runs through non-frozen arrows of $Q$. Note that the homology of degree $0$ of the $2$-dimensional relative Ginzburg dg algebra is the completed relative preprojective algebra and the homology of degree $0$ of the $3$-dimensional relative Ginzburg dg algebra is the relative Jacobian algebra.
The {\em $d$-dimensional Ginzburg morphism} associated with the ice quiver with potential $(Q,F,W)$ is the morphism from the $(d-1)$-dimensional Ginzburg dg algebra associated with the quiver with potential $(F,0)$ to the $d$-dimensional relative Ginzburg dg algebra associated with the ice quiver with potential $(Q,F,W)$ which maps $e_i$ to $e_i$ and $\alpha$ to $\alpha$ and $\alpha^*$ to $-\del_{\alpha}W$ and $t_i$ to $\sum_\alpha e_i(\alpha \alpha^*-\alpha^* \alpha)e_i$, where $\alpha$ runs through non-frozen arrows of $Q$.

\subsection{Ginzburg--Lazaroiu morphisms} \label{ss:Ginzburg-Lazaroiu morphisms}

We introduce a special class of augmented morphisms between pseudocompact dg algebras. From now on, we always suppose that the field $k$ is of characteristic $0$. Let $l_{\ol{A}}$ and $l_B$ be finite-dimensional semisimple $k$-algebras. Denote their product by $l_A$ and the canonical algebra injection $l_B \to l_A$ by $\varphi$.
Let $\si_{\ol{A}}$ and $\si_B$ be the Casimir elements associated with given traces on 
$l_{\ol{A}}$ respectively $l_B$. Denote their sum by $\si_A$. Let $d\geq 2$ be an integer. For a dg algebra $A$, we use the notation $\Tr(A)$ for the quotient complex $A/[A,A]$ of $A$ by the
subcomplex generated by graded commutators. Suppose that we are given a quintuple $(N,F,\eta,w_A,w_B)$ satisfying Assumptions~\ref{ass:quintuple} parts~a), b) and c) below.
\begin{assumptions} \label{ass:quintuple}\mbox{}
\begin{itemize}
\item[a)] $F$ is a pseudocompact graded $l_B$-bimodule of finite total dimension concentrated in degrees $[\frac{3-d}{2},0]$ and $N$ is a pseudocompact graded $l_A$-bimodule of finite total dimension concentrated in degrees $[2-d,0]$.
\item[b)] $\eta$ is a non-degenerate and graded anti-symmetric element of $N\ten_{l^e_A}N$ which is of degree $2-d$.
\end{itemize}
\end{assumptions}
\addtocounter{theorem}{-1}
We define $R=\Si^{d-3}DF$ and denote by $\eta_B$ the image of the identity $\id_{F}$ under the composed map
\[
\begin{tikzcd}
\Hom_k(F,F) & F\ten_k DF \arrow[swap]{l}{\sim} \arrow{r} & (F\oplus R)\ten_{l^e_B}(F\oplus R)\: ,
\end{tikzcd}
\]
where the first map is the canonical graded $k$-linear bijection and the second map maps $a\ten b$ to
\[
(-1)^{(d-3)|a|}a\ten s^{d-3}b-(-1)^{|a||b|}s^{d-3}b\ten a\: .
\]
Clearly, the element $\eta_B$ is non-degenerate and graded anti-symmetric of degree $3-d$.
\begin{assumptions} (continued)
\begin{itemize}
\item[c)] $w_B$ is an element of $\Tr(T_{l_B}(F\oplus R))$ which is of degree $4-d$ such that we have $\{w_B, w_B\}_{\omega_{\eta_B}}=0$ (\cf section~\ref{ss:the necklace bracket}), and $w_A$ is an element of $\Tr(T_{l_A}(F\oplus N))$ which is of degree $3-d$ such that 
the element $w_B+\{w_B,w_A\}_{\omega_{\eta_B}}+\frac{1}{2}\{w_A,w_A\}_{\omega_\eta}$ lies in the kernel of the canonical surjection $\Tr(T_{l_A}(F\oplus R\oplus N)) \to \Tr(T_{l_A}(F\oplus N))$.
\end{itemize}
\end{assumptions}
\begin{remark} We expect that the conditions in part~c) have a deformation-theoretic interpretation in the spirit of \cite{ThanhofferVandenBergh18}.
\end{remark}

The {\em $d$-dimensional Ginzburg--Lazaroiu morphism} associated with the above quintuple $(N,F,\eta,w_A,w_B)$ is the $\varphi$-augmented morphism
\[
\gamma \colon (T_{l_B}(F\oplus R\oplus z_B\, l_B),d)\longrightarrow (T_{l_A}(F\oplus N\oplus z_A\, l_{\ol{A}}),d)
\]
between pseudocompact dg algebras, where $z_B$ is an $l_B$-central indeterminate of degree $2-d$ and $z_A$ is an $l_A$-central indeterminate of degree $1-d$ whose annihilator is $l_B$. The topology of both pseudocompact graded algebras are the product topology. The unit and the augmentation maps of both augmented pseudocompact graded algebras are the natural ones. The differential of $T_{l_B}(F\oplus R\oplus z_B\, l_B)$ is determined by
\[
d(v)=\{w_B,v\}_{\omega_{\eta_B}}\mbox{ for all }v\in F\oplus R\quad \mbox{and}\quad d(z_B)=\si_B'\, \eta_B\, \si_B''\: .
\]
The differential of $T_{l_A}(F\oplus N\oplus z_A\, l_{\ol{A}})$ is determined by
\[
d(v)=\{w_B,v\}_{\omega_{\eta_B}}\mbox{ for all }v\in F\ko d(v)=\{w_A,v\}_{\omega_\eta}\mbox{ for all }v\in N\quad\mbox{and}\quad d(z_A)=\si_{\ol{A}}'\,\eta \,\si_{\ol{A}}''\: .
\]
Note that the differential of $v\in F$ lies in $T_{l_A}(F)$ for degree reasons.
The morphism $\gamma$ is determined by
\[
\gamma(v)=v\mbox{ for all }v\in F\ko \gamma(v)=-\{w_A,v\}_{\omega_{\eta_B}}\mbox{ for all }v\in R\quad \mbox{and}\quad \gamma(z_B)=\si_B'\, \eta \, \si_B''\: .
\]
Many examples of Ginzburg--Lazaroiu morphisms arise as deformed relative Calabi--Yau completions as introduced in \cite{Yeung16, Wu23a}. To check the algebras $(T_{l_A}(F\oplus N\oplus z_A\, l_{\ol{A}}),d)$ and $(T_{l_B}(F\oplus R\oplus z_B\, l_B),d)$ are honest dg algebras and the morphism $\gamma$ is an honest morphism between dg algebras, we need the following propositions and lemmas.

\begin{proposition} \label{prop:necklace b}
We have $d^2(v)=0$ in $(T_{l_B}(F\oplus R\oplus z_B\, l_B),d)$ for all $v$ in $F\oplus R$ if and only if we have $\{w_B,w_B\}_{\omega_{\eta_B}}=0$.
\end{proposition}
\begin{proof}
The sufficiency follows from the equality~(9.6) of \cite{VandenBergh15} and the necessity follows from the proof of Lemma~10.5 of \cite{VandenBergh15}.
\end{proof}

Note that if the equivalent conditions of the above Proposition~\ref{prop:necklace b} hold, 
then we have $d^2(v)=0$ in $(T_{l_A}(F\oplus N\oplus z_A\, l_{\ol{A}}),d)$ for all $v$ in $F$.

Recall that a graded left Loday algebra (\cf part~(3) of Proposition~1.4 of \cite{VandenBergh08a} for the terminology and section~1 of \cite{Loday93} for the definition) is a graded vector space $A$ endowed with a binary operation $\{?, -\}$ of degree $n$ satisfying
\[
\{a, \{b, c\}\}=\{\{a, b\}, c\}+(-1)^{(|a|+n)(|b|+n)}\{b, \{a, c\}\}
\]
for all $a$, $b$, $c\in A$. For the proof of the following two propositions, let us define a graded left Loday algebra structure on $T_{l_A}(F\oplus \Si R\oplus N)$ as follows:
we denote by $\eta_B^\Si$ the image of the identity $\id_{F}$ under the composed map
\[
\begin{tikzcd}
 \Hom_k(F,F) & F\ten_k F^* \arrow[swap]{l}{\sim} \arrow{r} & (F\oplus \Si R)\ten_{l^e_B}(F\oplus \Si R)\: ,
 \end{tikzcd}
\]
where the first map is the canonical graded $k$-linear bijection and the second map maps $a\ten b$ to
\[
(-1)^{(d-2)|a|}a\ten s^{d-2}b-(-1)^{|a||b|}s^{d-2}b\ten a\: .
\]
Clearly, the element $\eta_B^\Si$ is non-degenerate and graded anti-symmetric of degree $2-d$.
By Proposition~A.3.3 of \cite{VandenBergh08a}, the map $\ldb?,-\rdb_{\omega_{\eta_B^\Si+\eta}}$ is a double Poisson bracket on \linebreak $T_{l_A}(F\oplus \Si R\oplus N)$. By Corollary~2.4.4 of \cite{VandenBergh08a}, the bracket $\{?,-\}_{\omega_{\eta_B^\Si+\eta}}$ makes $T_{l_A}(F\oplus \Si R\oplus N)$ into a graded left Loday algebra.

\begin{proposition} \label{prop:necklace a1}
We have $d^2(v)=0$ in $(T_{l_A}(F\oplus N\oplus z_A\, l_{\ol{A}}),d)$ for all $v$ in $N$ if and only if the element $\{w_B,w_A\}_{\omega_{\eta_B}}+\frac{1}{2}\{w_A,w_A\}_{\omega_\eta}$ lies in the image of the canonical injection $\Tr(T_{l_A}F) \to \Tr(T_{l_A}(F\oplus N))$.
\end{proposition}
\begin{proof}
Recall that $T_{l_A}(F\oplus \Si R\oplus N)$ is a graded left Loday algebra when endowed with the bracket $\{?,-\}_{\omega_{\eta_B^\Si+\eta}}$ and similarly for $T_{l_A}(F\oplus N)$ with the bracket $\{?,-\}_{\omega_\eta}$. Therefore, for any $v\in N$, we have
\begin{align*}
d^2(v) &=d(\{w_A,v\}_{\omega_\eta}) \\
       &=\{w_B,\{w_A,v\}_{\omega_\eta}\}_{\omega_{\eta_B}}+\{w_A,\{w_A,v\}_{\omega_\eta}\}_{\omega_\eta} \\
       &=\{\{w_B,w_A\}_{\omega_{\eta_B}},v\}_{\omega_\eta}+\frac{1}{2}\{\{w_A,w_A\}_{\omega_\eta},v\}_{\omega_\eta} \\
       &=\{\{w_B,w_A\}_{\omega_{\eta_B}}+\frac{1}{2}\{w_A,w_A\}_{\omega_\eta},v\}_{\omega_{\eta}} \: .
\end{align*}
This implies the sufficiency. The necessity follows from the same argument as in the proof of Lemma~10.5 of \cite{VandenBergh15}.
\end{proof}

\begin{proposition} \label{prop:necklace a2}
We have $(d\circ \gamma)(v)=(\gamma \circ d)(v)$ for all $v$ in $R$ if and only if the image of $w_B+\{w_B,w_A\}_{\omega_{\eta_B}}+\frac{1}{2}\{w_A,w_A\}_{\omega_\eta}$ under the canonical surjection
\[
\begin{tikzcd}
\Tr(T_{l_A}(F\oplus R\oplus N)) \arrow{r} & \Tr(T_{l_A}(F\oplus N))
\end{tikzcd}
\]
lies in the image of the canonical injection $\Tr(T_{l_A}N) \to \Tr(T_{l_A}(F\oplus N))$.
\end{proposition}
\begin{proof}
Recall that $T_{l_A}(F\oplus \Si R\oplus N)$ is a graded left Loday algebra when endowed with the bracket $\{?,-\}_{\omega_{\eta_B^\Si+\eta}}$ and similarly for $T_{l_A}(F\oplus R\oplus N)$ with the bracket $\{?,-\}_{\omega_{\eta_B}}$. Therefore, for any $v\in R$, we have
\begin{align*}
(d\circ \gamma)(v) &=d(-\{w_A,v\}_{\omega_{\eta_B}}) \\
                   &=-\{w_B,\{w_A,v\}_{\omega_{\eta_B}}\}_{\omega_{\eta_B}}-\{w_A,\{w_A,v\}_{\omega_{\eta_B}}\}_{\omega_\eta} \\
                   &=-\{\{w_B,w_A\}_{\omega_{\eta_B}},v\}_{\omega_{\eta_B}}-\{w_A,\{w_B,v\}_{\omega_{\eta_B}}\}_{\omega_{\eta_B}}-\frac{1}{2}\{\{w_A,w_A\}_{\omega_\eta},v\}_{\omega_{\eta_B}} \: .
\end{align*}
Clearly, the tensor algebra $T_{l_B}(F\oplus R)$ admits a unique $\N$-grading such that the elements of $F$ are of degree $0$ and the elements of $R$ are of degree $1$. We define the $R$-degree to be the degree with respect to this grading. The vector space $\Tr(T_{l_B}(F\oplus R))$ inherits the $R$-grading.
If $w_B$ is homogeneous of $R$-degree $0$, then we have
\[
(\gamma \circ d)(v)=\gamma(\{w_B,v\}_{\omega_{\eta_B}})=\{w_B,v\}_{\omega_{\eta_B}}
\]
for all $v$ in $R$. If $w_B$ is homogeneous of $R$-degree at least $1$, then
by considering the internal degree we see that it must be
homogeneous of $R$-degree $1$. In this case, we have
\[
(\gamma \circ d)(v)=\gamma(\{w_B,v\}_{\omega_{\eta_B}})=-\{w_A,\{w_B,v\}_{\omega_{\eta_B}}\}_{\omega_{\eta_B}}
\]
for all $v$ in $R$. In conclusion,
the element $(\gamma \circ d)(v)$ must equal the component of $R$-degree $0$ of
$\{w_B,v\}_{\omega_{\eta_B}}-\{w_A,\{w_B,v\}_{\omega_{\eta_B}}\}_{\omega_{\eta_B}}$ for the general case. Therefore, we have
\[
(d\circ \gamma)(v)=(\gamma \circ d)(v)
\]
for all $v$ in $R$ if and only if the component of $R$-degree $0$ of
\[
\{w_B+\{w_B,w_A\}_{\omega_{\eta_B}}+\frac{1}{2}\{w_A,w_A\}_{\omega_\eta},v\}_{\omega_{\eta_B}}
\]
vanishes for all $v$ in $R$. This implies the sufficiency. The necessity follows from the same argument as in the proof of Lemma~10.5 of \cite{VandenBergh15}.
\end{proof}

Note that the equivalent conditions in Propositions~\ref{prop:necklace b}, \ref{prop:necklace a1}, \ref{prop:necklace a2} hold if and only if the conditions in part~c) of the Assumptions~\ref{ass:quintuple} on the quintuple $(N,F,\eta,w_A,w_B)$ hold. If we have $d\leq 3$, then $w_B$ vanishes and the condition on $w_A$ automatically holds.

\begin{lemma}
We have $d^2(z_A)=0$ and $d^2(z_B)=0$.
\end{lemma}
\begin{proof}
The second equality follows from section~9.2 of \cite{VandenBergh15}. To prove the first one, note that we have
\[
d^2(z_A)=d(\si_{\ol{A}}'\, \eta \, \si_{\ol{A}}'')=\{w_A,\si_{\ol{A}}'\, \eta \, \si_{\ol{A}}''\}_{\omega_\eta}\: .
\]
By Lemma~9.2 of \cite{VandenBergh15} (which also holds for $w_A\in T_{l_A}(F\oplus N\oplus z_A\, l_{\ol{A}})$), we have
\[
\{w_A,\si_A'\, \eta\, \si_A''\}_{\omega_\eta}=0 \: .
\]
If we multiply by $1_{l_{\ol{A}}}$ from both sides, we obtain $\{w_A,\si_{\ol{A}}'\, \eta \, \si_{\ol{A}}''\}_{\omega_\eta}=0$.
\end{proof}

For a graded tensor algebra $T_l V$, we denote by $\sym \colon \Tr(T_l V)\to (T_l V)_l$ the cyclic symmetrisation map which vanishes on $l$ and maps the class of an element $a_1\ten\cdots\ten a_n$, where $a_i$ lie in $V$, to the element
\[
\sum_{1\leq i\leq n}\pm a_i\ten\cdots\ten a_n\ten a_1\ten\cdots\ten a_{i-1}\: .
\]
Here the signs are given by the Koszul sign rule. Denote by $\ol{\eta_B}$ the image of $\eta_B$ under the canonical projection to $F\ten_{l^e_B}R$. It is non-degenerate.

\begin{lemma}
We have $(d\circ \gamma)(v)=(\gamma \circ d)(v)$ for all $v$ in $F$ and $(d\circ \gamma)(z_B)=(\gamma \circ d)(z_B)$.
\end{lemma}
\begin{proof}
The first statement is clear. To prove the second one, since the component of each tensor degree of the element
\[
\sym(w_A)=-\ol{\eta_B}'\{w_A,\ol{\eta_B}''\}_{\omega_{\eta_B}}+(-1)^{|\eta'|+1}\eta' \{w_A,\eta''\}_{\omega_\eta}
\]
of $(T_{l_A}(F\oplus N))_{l_A}$ is stable under the corresponding cyclic permutation group, we have
\[
\sym(w_A)=-(-1)^{|\ol{\eta_B}'||\ol{\eta_B}''|}\{w_A,\ol{\eta_B}''\}_{\omega_{\eta_B}}\ol{\eta_B}'+(-1)^{|\eta'||\eta''|+1}\{w_A,\eta''\}_{\omega_\eta}\eta'\: .
\]
Let us show that the difference of the right hand sides equals $\gamma(\eta_B)-d(\eta)$, which therefore has to vanish. Indeed, the difference equals
\begin{align*}
 & -\ol{\eta_B}'\{w_A,\ol{\eta_B}''\}_{\omega_{\eta_B}}+(-1)^{|\ol{\eta_B}'||\ol{\eta_B}''|}\{w_A,\ol{\eta_B}''\}_{\omega_{\eta_B}}\ol{\eta_B}' \\
 & -(-1)^{|\eta'||\eta''|+1}\{w_A,\eta''\}_{\omega_\eta}\eta'+(-1)^{|\eta'|+1}\eta' \{w_A,\eta''\}_{\omega_\eta} \: .
\end{align*}
The first line equals $\gamma(\eta_B)$ because we have
\begin{align*}
\gamma(\eta_B) & =\gamma(\ol{\eta_B}'\ol{\eta_B}''-(-1)^{|\ol{\eta_B}'||\ol{\eta_B}''|}\ol{\eta_B}''\ol{\eta_B}') \\
               & =-\ol{\eta_B}'\{w_A,\ol{\eta_B}''\}_{\omega_{\eta_B}}+(-1)^{|\ol{\eta_B}'||\ol{\eta_B}''|}\{w_A,\ol{\eta_B}''\}_{\omega_{\eta_B}}\ol{\eta_B}' \: .
\end{align*}
The second line equals $-d(\eta)$ because we have
\begin{align*}
d(\eta) & =d(\eta' \eta'') \\
        & =d(\eta')\eta''+(-1)^{|\eta'|}\eta'd(\eta'') \\
        & =(-1)^{|\eta'||\eta''|+1}d(\eta'')\eta'+(-1)^{|\eta'|}\eta'd(\eta'') \\
        & =(-1)^{|\eta'||\eta''|+1}\{w_A,\eta''\}_{\omega_\eta}\eta'+(-1)^{|\eta'|}\eta' \{w_A,\eta''\}_{\omega_\eta} \: .
\end{align*}
Therefore, we have
\[
(d\circ \gamma)(z_B)=d(d(z_A)+\gamma(z_B))=d(\si_A'\, \eta \, \si_A'')=\gamma(\si_A'\, \eta_B \, \si_A'')=(\gamma \circ d)(z_B)\: .
\]
\end{proof}

Our intuition comes from the {\em ice quiver case}, \ie the special case of the
above setting when the algebras $l_A=\prod_{i=1}^n ke_i$ and $l_B=\prod_{i=1}^m ke_i$ are finite products of copies of the ground field $k$ for some integers $n\geq m$. We choose the Casimir elements $\si_A=\sum_{i=1}^n e_i\ten e_i$ and $\si_B=\sum_{i=1}^m e_i\ten e_i$. Since the element $\eta_B$ lies in $(F\ten_{l_B^e}R)\oplus (R\ten_{l_B^e}F)$ and is non-degenerate and graded anti-symmetric in $(F\oplus R)\ten_{l_B^e}(F\oplus R)$, we can write the element $\eta_B$ as the sum $\sum_{1\leq i,j\leq m}[y_{ij}^t,y_{ji}^{t*}]$ for a suitable homogeneous $k$-basis $(y_{ij}^t)$ of each graded $k$-vector space $e_i F e_j$ and $(y_{ji}^{t*})$ of each graded $k$-vector space $e_i R e_j$. Since the element $\eta$ is non-degenerate and graded anti-symmetric in $N\ten_{l_A^e}N$, we can write the element $\eta$ as the sum $\sum_{1\leq i,j\leq n}[x_{ij}^t,x_{ji}^{t*}]$ for a suitable homogeneous $k$-basis $(x_{ij}^t,x_{ji}^{t*})$ of each graded $k$-vector space $e_i N e_j$ such that $x_{ij}^t$ are of degree greater than or equal to $\frac{2-d}{2}$ and the elements $x_{ij}^t$ and $x_{ji}^{t*}$ are all distinct unless $i$ coincides with $j$ and $d$ is divisible by $4$, in which case we may have $x_{ij}^t=x_{ji}^{t*}$. Then we regard $i$ as a frozen vertex for $1\leq i\leq m$ and as a non-frozen vertex for $m<i\leq n$. We regard the above homogeneous $k$-basis elements $y_{ij}^t$ as the original frozen arrows from $j$ to $i$ (the symbol $F$ stands for the graded vector space spanned by the original frozen arrows), the elements $y_{ji}^{t*}$ as the reversed frozen arrows from $j$ to $i$ (the symbol $R$ stands for the graded vector space spanned by the reversed frozen arrows), the elements $x_{ij}^t$ as the original non-frozen arrows from $j$ to $i$ and the elements $x_{ji}^{t*}$ which are distinct from $x_{ij}^t$ as the reversed non-frozen arrows from $j$ to $i$ (the symbol $N$ stands for the graded vector space spanned by the original and the reversed non-frozen arrows). We regard $z_A$ as the sum of loops at the non-frozen vertices and $z_B$ as the sum of loops at the frozen vertices. We regard $w_A$ as the potential on the relative double (only double the non-frozen part) of the whole quiver and $w_B$ as the potential on the double of the frozen subquiver. The necklace brackets $\{?, -\}_{\omega_{\eta}}$ and $\{?, -\}_{\omega_{\eta_B}}$ are precisely the {\em necklace brackets} on the completed dg path algebras associated with these two quivers respectively, \cf section~2 of \cite{BocklandtBruyn02}.

Similarly, let $l_{\ol{A}}$, $l_{B_1}$, $l_{B_2}$ be finite-dimensional semisimple $k$-algebras. Denote their product by $l_A$ and the canonical algebra injection $l_{B_i} \to l_A$ by $\varphi_i$, where $i=1$, $2$. Suppose that we are given a $7$-tuple $(N,F_1,F_2,\eta,w_A,w_{B_1},w_{B_2})$ satisfying the assumptions in analogy with the Assumptions~\ref{ass:quintuple} parts~a), b) and c). We define the associated {\em $d$-dimensional Ginzburg--Lazaroiu cospan} to be the cospan
\[
\begin{tikzcd}
    & (T_{l_{B_1}}(F_1\oplus R_1\oplus z_{B_1}l_{B_1}),d)\arrow{d}{\gamma_1} \\
  (T_{l_{B_2}}(F_2\oplus R_2\oplus z_{B_2}l_{B_2}),d)\arrow{r}{\gamma_2} & (T_{l_A}(F_1\oplus F_2\oplus N\oplus z_A\, l_{\ol{A}}),d)
\end{tikzcd}
\]
for pseudocompact dg algebras augmented over
\[
\begin{tikzcd}
   & l_{B_1}\arrow{d}{\varphi_1} \\
  l_{B_2}\arrow{r}{\varphi_2} & l_A
\end{tikzcd}
\]
in analogy with the $d$-dimensional Ginzburg--Lazaroiu morphism associated with a quintuple.

\subsection{The main results} \label{ss:the main results}

Let $k$ be a field of characteristic $0$. Let $l_A$ and $l_B$ be finite-dimensional semisimple $k$-algebras. Let $\varphi \colon l_B \to l_A$ be a morphism between algebras (which does not necessarily preserve the unit!) such that the equivalent conditions in Proposition~\ref{prop:decomposition} hold. Without loss of generality, we may and will assume that we have $l_A=l_{\ol{A}}\times l_B$ and $\varphi$ is the canonical algebra injection. Recall that, by Corollary~\ref{cor:symplectic}, a Calabi--Yau structure yields a canonical symplectic form on the $(-1)$-shifted dual of the graded Yoneda algebra of any perfectly valued dg module.

\begin{theorem} \label{thm:main}
Let $A$ and $B$ be complete pseudocompact dg algebras concentrated in non-positive degrees augmented over $l_A$ respectively $l_B$. Let $f\colon B \to A$ be a $\varphi$-augmented morphism between pseudocompact dg algebras. Let $d\geq 2$ be an integer. Then the following are equivalent.
\begin{itemize}
\item[i)] In the model category $\PCAlgc \varphi$ of $\varphi$-augmented morphisms between complete pseudocompact dg algebras, the morphism $f\colon B \to A$ is weakly equivalent to the \linebreak $d$-dimensional Ginzburg--Lazaroiu morphism $\gamma$ associated with a quintuple
\[
(N,F,\eta,w_A,w_B)\: ,
\]
where the elements $w_A$ and $w_B$ only contain cubic and higher terms.
\item[ii)] The morphism $f\colon B \to A$ carries a relative $d$-Calabi--Yau structure. Moreover, the kernel of the induced map $\Si^{-1}\D\Ext^*_B(l_B, l_B) \to  \Si^{-1}\D\Ext^*_A(l_A, l_A)$ is a Lagrangian (homogeneous) subspace concentrated in degrees less than or equal to $\frac{3-d}{2}$.
\end{itemize}
\end{theorem}

\begin{remarks}
\begin{itemize}
\item[a)] Note that the degree condition in part~ii) is vacuous if the Calabi--Yau dimension $d$ is less than or equal to $3$.
\item[b)] In the absolute case (when $B$ vanishes), the theorem reduces to the main theorem of
Van den Bergh's \cite{VandenBergh15}.
\item[c)] In the setting of non-pseudocompact dg algebras augmented over finite products of copies of $k$, and for special graded $l_A$-bimodules $N$ and potentials $w_A$ and $w_B$, the implication from i) to ii) is due to Yeung \cite{Yeung16}.
\end{itemize}
\end{remarks}

In the following two sections, we will prove the two implications in the theorem. The reader may find the proof long and technical but compared to Joyce--Safronov's proof \cite{JoyceSafronov19} of the corresponding result in the commutative case, it is relatively short and involves
few computations.

Before embarking on the proof, we state two important special cases and a generalisation.

\begin{corollary} \label{cor:dimension 3} Suppose that
\begin{itemize}
\item[a)] the assumptions in part~ii) of the above Theorem~\ref{thm:main} hold, 
\item[b)] we have $d=3$ and
\item[c)] the algebras $l_A$ and $l_B$ are finite products of copies of the 
ground field $k$.
\end{itemize}
Then $f$ is weakly equivalent to a $3$-dimensional Ginzburg morphism. Moreover, if the graded algebras $A$ and $B$ are concentrated in degree $0$, then $f$ is isomorphic to a morphism from a completed preprojective algebra to a relative Jacobian algebra.
\end{corollary}
\begin{proof}
This is the ice quiver case considered in section~\ref{ss:Ginzburg-Lazaroiu morphisms}. We now use the notation from that section. Since we have $d=3$, the elements $x_{ij}^t$ and $x_{ji}^{t*}$ in a suitable homogeneous $k$-basis $(x_{ij}^t,x_{ji}^{t*})$ are all distinct. The elements $x_{ij}^t$ are of degree $0$ and the $x_{ji}^{t*}$ are of degree $-1$. Let $Q$ be the quiver whose arrows from $j$ to $i$ are the $x_{ij}^t$ and $F\subseteq Q$ the frozen subquiver whose arrows from $j$ to $i$ are the $y_{ij}^t$. In this case, the potential $w_B$ vanishes and, for degree reasons, the potential $W=w_A$ only contains arrows in the given quiver $Q$. Therefore, by the implication from ii) to i) in Theorem~\ref{thm:main}, the morphism $f$ is weakly equivalent to the $3$-dimensional Ginzburg morphism $\gamma$ associated with the ice quiver with potential $(Q,F,W)$. If the graded algebras $A$ and $B$ are concentrated in degree $0$, then the morphism $f$ is isomorphic to $H^0(\gamma)$. This implies the statement.
\end{proof}

\begin{corollary} \label{cor:dimension 2} Suppose that
\begin{itemize}
\item[a)] the assumptions in part~ii) of the above Theorem~\ref{thm:main} hold, 
\item[b)] we have $d=2$ and
\item[c)] the algebras $l_A$ and $l_B$ are finite products of copies of the 
ground field $k$.
\end{itemize}
Then $f$ is weakly equivalent to a $2$-dimensional Ginzburg morphism. Moreover, if the graded algebras $A$ and $B$ are concentrated in degree $0$, then $f$ is isomorphic to a morphism from a finite product of copies of the power series algebra $k\llbracket x \rrbracket$ to a completed relative preprojective algebra.
\end{corollary}
\begin{proof}
Similar to the proof of the preceding corollary. In this case, the graded $l_B$-bimodule $F$ vanishes and the graded $l_A$-bimodule $N$ is concentrated in degree $0$. Moreover, both elements $w_A$ and $w_B$ vanish.
\end{proof}

Let $l_A$, $l_{B_1}$, $l_{B_2}$ be finite-dimensional semisimple $k$-algebras. Let $\varphi_1 \colon l_{B_1} \to l_A$ and $\varphi_2 \colon l_{B_2} \to l_A$ be morphisms between algebras (which do not necessarily preserve the unit!) such that the equivalent conditions in Proposition~\ref{prop:decomposition} hold and the products $\varphi_1(1_{B_1})\cdot \varphi_2(1_{B_2})$ and $\varphi_2(1_{B_2})\cdot \varphi_1(1_{B_1})$ equal zero. Without loss of generality, we may and will assume that we have $l_A=l_{\ol{A}}\times l_{B_1}\times l_{B_2}$ and $\varphi_1$, $\varphi_2$ are the canonical algebra injection.

\begin{theorem} \label{thm:CY-cospans}
Let $A$, $B_1$, $B_2$ be complete pseudocompact dg algebras concentrated in non-positive degrees augmented over $l_A$, $l_{B_1}$, $l_{B_2}$, respectively. Let $f_i\colon B_i \to A$ be a \linebreak $\varphi_i$-augmented morphism between pseudocompact dg algebras, where $i=1$, $2$. Let $d\geq 2$ be an integer. Then the following are equivalent.
\begin{itemize}
\item[i)] In the model category of cospans for complete pseudocompact dg algebras augmented over
\[
\begin{tikzcd}
   & l_{B_1}\arrow{d}{\varphi_1} \\
  l_{B_2}\arrow{r}{\varphi_2} & l_A\mathrlap{\: ,}
\end{tikzcd}
\]
the cospan
\[
\begin{tikzcd}
   & B_1\arrow{d}{f_1} \\
  B_2\arrow{r}{f_2} & A
\end{tikzcd}
\]
is weakly equivalent to the $d$-dimensional Ginzburg--Lazaroiu cospan associated with a $7$-tuple $(N,F_1,F_2,\eta,w_A,w_{B_1},w_{B_2})$, where the elements $w_A$, $w_{B_1}$, $w_{B_2}$ only contain cubic and higher terms.
\item[ii)] The cospan
\[
\begin{tikzcd}
   & B_1\arrow{d}{f_1} \\
  B_2\arrow{r}{f_2} & A
\end{tikzcd}
\]
carries a $d$-Calabi--Yau structure. Moreover, the kernel of the induced map
\[
\Si^{-1}\D\Ext^*_{B_1\times B_2}(l_{B_1}\times l_{B_2}, l_{B_1}\times l_{B_2}) \to  \Si^{-1}\D\Ext^*_A(l_A, l_A)
\]
is a Lagrangian (homogeneous) subspace concentrated in degrees less than or equal to $\frac{3-d}{2}$.
\end{itemize}
\end{theorem}
\begin{proof}
The statement follows from Proposition~\ref{prop:CY-cospan} and Theorem~\ref{thm:main}.
\end{proof}

\subsection{Proof of the implication from i) to ii) in Theorem~\ref{thm:main}} \label{ss:proof from i) to ii)}

The following proof is inspired by the ice quiver case in section~\ref{ss:Ginzburg-Lazaroiu morphisms} and we advise the reader to refer to that section constantly in order to follow the reasoning. We first construct an exact relative \linebreak $d$-Calabi--Yau structure on $f$. Denote the pseudocompact graded $l_A$-bimodule $F\oplus N\oplus z_A\, l_{\ol{A}}$ by $V_A$ and the pseudocompact graded $l_B$-bimodule $F\oplus R\oplus z_B\, l_B$ by $V_B$. Since they are of finite total dimension, the source $(T_{l_B}V_B,d)$ and the target $(T_{l_A}V_A,d)$ of the Ginzburg--Lazaroiu morphism $\gamma$ are smooth. By sections~12.1 and 12.4 of \cite{VandenBergh15}, the dg algebra $A$ is quasi-isomorphic to $(T_{l_A}V_A,d)$ and similarly for $B$. Thus, since smoothness is preserved under quasi-isomorphisms (by the relative version of part~(d) of Proposition~3.10 of \cite{Keller11b}, it is even preserved under localisations), the pseudocompact dg algebras $A$ and $B$ are smooth. Moreover, since quasi-isomorphisms induce equivalences of derived categories, exact relative $d$-Calabi--Yau structures are also preserved under quasi-isomorphisms. So we may and will assume that we have $f=\gamma$. Let us put
\[
X(A)=\cone(\Omega^1_{l_A}A \to A\ten_{l_A}A)
\quad \mbox{and}\quad
X(B)=\cone(\Omega^1_{l_B}B \to B\ten_{l_B}B)\: .
\]
These are cofibrant resolutions of $A$ and $B$ as pseudocompact dg modules over $A^e$ respectively $B^e$. Recall that we denote the inverse of the bijection in Proposition~\ref{prop:invariant space} by $\dag$. Since the element $d(z_B^\dagger)=\eta_B$ is graded anti-symmetric in $(F\oplus R)\ten_{l^e_B}(F\oplus R)$ and the element $d(z_A^\dagger)+f(z_B^\dagger)=\eta$ is graded anti-symmetric in $N\ten_{l_A^e}N$, the element $(\ol{z_A}^\dagger,s\ol{z_B}^\dagger)$ in
\[
\cone(f\colon (B/(l_B+[B,B]))_{l_B} \to (A/(l_A+[A,A]))_{l_A})
\]
is closed. By the relative version of Proposition~\ref{prop:connecting morphism}, the corresponding class $[(\ol{z_A}^\dagger,s\ol{z_B}^\dagger)]$ in the reduced relative cyclic homology
\[
H^{1-d}(\cone(f\colon (B/(l_B+[B,B]))_{l_B} \to (A/(l_A+[A,A]))_{l_A}))=\HC^{\red}_{d-1}(A,B)
\]
is mapped to $[((0,-sDz_A^\dagger),s(0,sDz_B^\dagger))]$ in $\HH^{\red}_d(A,B)$ by Connes' map $B$. We define
\[
\xi_B=(1\ten1)\ten sDz_B+sDz_B\ten(1\ten1)-(sD\ten sD)(d(z_B)) \in X(B)\ten_{B^e}X(B)
\]
and
\[
\xi_A=-(1\ten 1)\ten sDz_A-sDz_A\ten(1\ten 1)+(sD\ten sD)(d(z_A)+f(z_B)) \in X(A)\ten_{A^e}X(A) \: .
\]
One can check that the pair $(\xi_A,s\xi_B)$ in the cone
\[
\cone(f\ten f\colon X(B)\ten_{B^e}X(B) \to X(A)\ten_{A^e}X(A))
\]
is a closed representative of the class $[((0,-sDz_A^\dagger),s(0,sDz_B^\dagger))]$. We claim that the class $[(\ol{z_A}^\dagger,s\ol{z_B}^\dagger)]$ in $\HC_{d-1}(A,B)$ gives an exact relative $d$-Calabi--Yau structure on $f$. To prove this, it suffices to prove that the morphisms $[\hat{\xi_B}s^{1-d}]$ and $[\hat{\xi}']$ defined in section~\ref{ss:Calabi-Yau structures} are isomorphisms. Equivalently, we need to prove that the class $[\xi_B]$ in
\[
H^{1-d}(B\lten_{B^e}B)\xlongleftarrow{_\sim} H^{1-d}(X(B)\ten_{B^e}X(B))
\]
is non-degenerate and that the class $[(\xi_A, (-1)^d s(f\ten\id)(\xi_B))]$ in
\[
\begin{tikzcd}
H^{-d}(A\lten_{A^e}\cone \mu) & H^{-d}(\cone(\id\ten f\colon X(A)\ten_{B^e}X(B)\to X(A)\ten_{A^e}X(A)))\arrow[swap]{l}{\sim}
\end{tikzcd}
\]
is non-degenerate. Since the element $w_B$ only contains cubic and higher terms, the condition~(1) in Lemma~10.2 of \cite{VandenBergh15} holds and hence the non-degeneracy of the class $[\xi_B]$ follows from the sufficiency in that lemma. For the non-degeneracy of the second class, we claim that it suffices to show that the morphism $\RHom_{A^e}([\hat{\xi}'],D l_A^e)$ in $\cd(l_A^e)$ is an isomorphism. Indeed, the dg algebra $A$ is complete $l_A$-augmented and pseudocompact so that the category $\cd(A^e)\op$ is compactly generated by $(\pvd A^e)\op$. Moreover, the category $\pvd A^e$ is generated by the object $D l_A^e$. Clearly, the claim follows. Since we have the natural isomorphism
\[
\begin{tikzcd}
\RHom_{A^e}(?,D l_A^e)\arrow[no head]{r}{\sim} & \D(l_A\lten_A\, ?\lten_A l_A)\: ,
\end{tikzcd}
\]
it suffices to show that the morphism $l_A\lten_A \Si^d (\cone \mu)^\vee\lten_A l_A\to l_A\lten_A A\lten_Al_A$ in $\cd(l_A^e)$ induced by $[\hat{\xi}']$ is an isomorphism. This holds if and only if the corresponding class in
\[
H^{-d}((l_A\lten_A A\lten_Al_A)\lten_{l_A^e}\RHom_{l_A^e}^{\pc}(l_A\lten_A (\cone \mu)^\vee\lten_A l_A,l_A^e))
\]
is non-degenerate. Since the pseudocompact dg $A^e$-modules $A$ and $A\lten_B A$ are perfect, so is $\cone \mu$. We have the canonical isomorphisms
\begin{align*}
\RHom_{l_A^e}^{\pc}(l_A\lten_A (\cone \mu)^\vee\lten_Al_A,l_A^e) & \simeq \RHom_{l_A^e}^{\pc}((\cone \mu)^\vee\lten_{A^e}l_A^e,l_A^e) \\
&\simeq\RHom_{A^e}^{\pc}((\cone \mu)^\vee,\RHom_{l_A^e}^{\pc}(l_A^e,l_A^e)) \\
&\liso l_A^e\lten_{A^e}(\cone \mu)^{\vee\vee} \\
&\liso l_A^e\lten_{A^e}\cone \mu \\
&\simeq l_A\lten_A\cone \mu \lten_A l_A\: .
\end{align*}
So it suffices to show that the class represented by the image of $\xi_A+(-1)^d s(f\ten\id)(\xi_B)$ in
\begin{equation} \label{eq:relative homology}
H^{-d}((l_A\lten_A A\lten_A l_A)\lten_{l_A^e}(l_A\lten_A\cone \mu \lten_A l_A))
\end{equation}
is non-degenerate. For this, let us first analyse $H^*(L \lten_{l_A^e} M)$, where we abbreviate 
\[
L=l_A\lten_A A\lten_A l_A \quad \mbox{and}\quad M=l_A\lten_A\cone \mu \lten_A l_A \: .
\]
Since the field $k$ is perfect and the $k$-algebra $l_A$ is semisimple, the $k$-algebra $l_A$ is separable and so is $l_A^e$. Thus, each object in the derived category $\cd(l_A^e)$ is isomorphic to its homology and the derived tensor product is isomorphic to the non-derived one. Clearly, we have the canonical isomorphism $L \iso l_A \oplus \Si V_A$. To analyse $M$, let us denote the functor $l_A\lten_A\, ?\lten_Al_A$ by $\Phi$. Recall that the morphism $\mu \colon A \lten_B A \to A$ in $\cd(A^e)$ corresponds to the morphism $f\colon B \to A$ in $\cd(B^e)$ by adjunction. Thus, the morphism $\Phi(\mu)$ is the canonical morphism $\nu$
\[
\begin{tikzcd}
(l_A\ten_{l_B}l_B\ten_{l_B}l_A)\oplus(l_A\ten_{l_B}\Si V_B\ten_{l_B}l_A)\arrow{r} & l_A \oplus\Si V_A
\end{tikzcd}
\]
in $\cd(l_A^e)$. Since the differentials of the source and the target of $\nu$ vanish, the homology of $\cone \nu$ is isomorphic to $\cok \nu \oplus \Si \ker \nu$. 
On the other hand, the kernel of $\nu$ is isomorphic to $\Si R \oplus \Si z_B\, l_B$ and the cokernel of $\nu$ is isomorphic to $l_{\ol{A}}\oplus \Si N\oplus\Si z_A\, l_{\ol{A}}$.
As a consequence of these observations, we obtain that the pseudocompact vector space~(\ref{eq:relative homology}) is canonically isomorphic to
\[
H^{-d}((l_A\oplus \Si V_A)\ten_{l_A^e}((l_{\ol{A}}\oplus \Si N\oplus\Si z_A\, l_{\ol{A}})\oplus \Si(\Si R\oplus\Si z_B\, l_B)))\: .
\]
On the other hand, we have
\begin{align*}
\xi_A+(-1)^d s(f\ten \id)(\xi_B)=&-(1\ten1)\ten sDz_A-sDz_A\ten(1\ten1) \\
 &+(sD\ten sD)(d(z_A)+f(z_B)) \\
 &+(-1)^d s((1\ten1)\ten sDz_B+sDf(z_B)\ten(1\ten1) \\
 &-(sD\ten sD)(f\ten \id)(d(z_B)))\: .
\end{align*}
Its image in
\[
(l_A\oplus \Si V_A)\ten_{l_A^e}((l_{\ol{A}}\oplus \Si N\oplus\Si z_A\, l_{\ol{A}})\oplus \Si(\Si R\oplus\Si z_B\, l_B))
\]
is
\[
-1\ten sz_A-sz_A\ten1+(s\ten s)(\si_A'\, \eta\si_A'')+(-1)^d (\id\ten s)(1\ten sz_B -(s\ten s)(\si_B'\, \ol{\eta_B}\si_B'')) \: .
\]
The pseudocompact graded $l_A$-bimodule $l_A \oplus \Si V_A$ decomposes as
\[
l_{\ol{A}}\oplus l_B \oplus \Si F\oplus \Si N\oplus \Si z_A\, l_{\ol{A}}\: .
\]
Since the elements
\[
1\ten sz_A^\dagger+sz_A^\dagger\ten 1\in (l_{\ol{A}} \oplus \Si z_A\, l_{\ol{A}})\ten_{l_{\ol{A}}^e}(l_{\ol{A}} \oplus \Si z_A\, l_{\ol{A}}) \ko \eta \in N\ten_{l_A^e}N \: ,
\]
\[
1\ten sz_B^\dagger \in l_B \ten_{l_B^e}\Si z_B\, l_B \quad \mbox{and} \quad \ol{\eta_B}\in F\ten_{l_B^e}R
\]
are non-degenerate, the class represented by the above element is non-degenerate.

It remains to show that the kernel of the induced map
\[
\Si^{-1}\D\Ext^*_B(l_B, l_B) \longrightarrow \Si^{-1}\D\Ext^*_A(l_A, l_A)
\]
is a Lagrangian subspace concentrated in degrees less than or equal to $\frac{3-d}{2}$. Since quasi-isomorphisms induce equivalences of derived categories, we have the commutative square
\[
\begin{tikzcd}
  \Ext^*_A(l_A,l_A)\arrow{r}\arrow[no head]{d}{\wr} & \Ext^*_B(l_B,l_B)\arrow[no head]{d}{\wr} \\
  \Ext^*_{(T_{l_A}V_A,d)}(l_A,l_A)\arrow{r} & \Ext^*_{(T_{l_B}V_B,d)}(l_B,l_B)
\end{tikzcd}
\]
of graded vector spaces, where the vertical maps are bijective and compatible with the bilinear forms. It suffices to show that the kernel of the induced map
\[
\begin{tikzcd}
\Si^{-1}\D\Ext^*_{(T_{l_B}V_B,d)}(l_B, l_B) \arrow{r} & \Si^{-1}\D\Ext^*_{(T_{l_A}V_A,d)}(l_A, l_A)
\end{tikzcd}
\]
has the same property and hence we may and will assume that we have $f=\gamma$. By Corollary~12.12 of \cite{VandenBergh15}, the pseudocompact graded $l_A$-bimodule $\Si^{-1}\D\Ext^*_A(l_A, l_A)$ is isomorphic to $\Si^{-1}l_A\oplus V_A$ and the pseudocompact graded $l_B$-bimodule $\Si^{-1}\D\Ext^*_B(l_B, l_B)$ is isomorphic to $\Si^{-1}l_B\oplus V_B$. Since the element $w_A$ only contains cubic and higher terms, the element $f(v)=-\{w_A, v\}_{\omega_{\eta_B}}$ lies in $\prod_{p\geq 2}V_A^{\ten_{l_A} p}$ for all $v\in R$. So the kernel of the above map is isomorphic to $R\oplus z_B\, l_B$, which is concentrated in degrees less than or equal to $\frac{3-d}{2}$ and whose total dimension is half of that of $\Si^{-1}l_B\oplus V_B$. Since the bilinear form on $\Ext_B^*(l_B, l_B)$ is induced by the multiplication $m_2$, by Lemma~\ref{lem:symplectic}, the $(-1)$-shifted dual bilinear form on $\Si^{-1}\D \Ext_B^*(l_B, l_B)$ is induced by the dual of the component $d_2$ of the differential. Since the element $d(z_B)^\dagger=\eta_B$ lies in $(F\ten_{l_B^e}R)\oplus(R\ten_{l_B^e}F)$, we deduce that the pseudocompact graded subspace $R\oplus z_B\, l_B$ is isotropic. We infer that $R\oplus z_B\, l_B$ is a Lagrangian subspace of $\Si^{-1}l_B \oplus V_B$. This concludes the proof of the implication from i) to ii) in Theorem~\ref{thm:main}.

\subsection{Proof of the implication from ii) to i) in Theorem~\ref{thm:main}} \label{ss:proof from ii) to i)}

If we take homology of the triangle
\[
\begin{tikzcd}
\HP^{\red}(B)\arrow{r}&\HP^{\red}(A)\arrow{r}&\HP^{\red}(A,B)\arrow{r}&\Si \HP^{\red}(B)
\end{tikzcd}
\]
of complexes and use Theorem~8.1 of \cite{VandenBergh15}, we deduce that we have $\HP^{\red}_*(A,B)=0$. Similar to the necessity of Corollary 8.3 of \cite{VandenBergh15}, one can prove that the morphism $f\colon B \to A$ is exact relative $d$-Calabi--Yau.

\emph{Step 1. We replace $A$ and $B$ by dg tensor algebras.}
Let $p_A\colon A'\to A$ be a cofibrant replacement in the model category $\PCAlgc l_A$. Its image
under the forgetful functor to $\PCAlgc l_B$ is still a cofibrant replacement (indeed, the model
category $\PCAlgc l_B$ is dual to the model category of cocomplete dg coalgebras and therefore its fibrations are the surjections, \cf section~1.3.1 of \cite{Lefevre03}).
By Corollary~12.11 of \cite{VandenBergh15}, there are weak equivalences $p_{A'}\colon (T_{l_A}V_A,d)\to A'$ in the model category $\PCAlgc l_A$ and $p_B\colon (T_{l_B}V_B,d)\to B$ in the model category $\PCAlgc l_B$ such that the induced differentials on the pseudocompact graded bimodules $V_A$ and $V_B$ over $l_A$ respectively $l_B$ vanish. By sections~12.1 and 12.4 of \cite{VandenBergh15}, the objects $(T_{l_A}V_A,d)$ in $\PCAlgc l_A$ and $(T_{l_B}V_B,d)$ in $\PCAlgc l_B$ are fibrant and cofibrant. So there is a morphism  $h\colon (T_{l_B}V_B,d)\to A'$ satisfying $p_A\circ h=f\circ p_B$, \cf the diagram below. By the necessity in Lemma~4.24 of \cite{DwyerSpalinski95}, there is a morphism $q_{A'}\colon A'\to (T_{l_A}V_A,d)$ such that the composed morphism $q_{A'}\circ p_{A'}$ is homotopic to the identity morphism $\id_{T_{l_A}V_A}$. In particular, the morphism $q_{A'}$ is a weak equivalence. It gives rise to the commutative diagram
\[
\begin{tikzcd}
  (T_{l_B}V_B,d)\arrow{r}{q_{A'}\circ h}\arrow[equal]{d} & (T_{l_A}V_A,d) \\
  (T_{l_B}V_B,d)\arrow{r}{h}\arrow[swap]{d}{p_B} & A'\arrow[swap]{u}{q_{A'}}\arrow{d}{p_A} \\
  B\arrow[swap]{r}{f} & A\mathrlap{\: ,}
\end{tikzcd}
\]
where all the vertical morphisms are weak equivalences. It follows that
the morphism $f\colon B \to A$ is weakly equivalent to $q_{A'}\circ h\colon (T_{l_B}V_B,d) \to (T_{l_A}V_A,d)$ in the model category $\PCAlgc \varphi$. By sections~12.1 and 12.4 of \cite{VandenBergh15}, the dg algebra $A$ is quasi-isomorphic to $(T_{l_A}V_A,d)$ and similarly for $B$. Thus, since smoothness is preserved under quasi-isomorphisms, the pseudocompact dg algebras $(T_{l_A}V_A,d)$ and $(T_{l_B}V_B,d)$ are smooth. Moreover, since quasi-isomorphisms induce equivalences of derived categories, exact relative $d$-Calabi--Yau structures are also preserved under quasi-isomorphisms and we have the commutative square
\[
\begin{tikzcd}
  \Ext^*_A(l_A,l_A)\arrow{r}\arrow[no head]{d}{\wr} & \Ext^*_B(l_B,l_B)\arrow{d}{\wr} \\
  \Ext^*_{(T_{l_A}V_A,d)}(l_A,l_A)\arrow{r} & \Ext^*_{(T_{l_B}V_B,d)}(l_B,l_B)
\end{tikzcd}
\]
of graded vector spaces, where the vertical maps are bijective and compatible with the bilinear forms. So the kernel of the induced map
\[
\begin{tikzcd}
\Si^{-1}\D\Ext^*_{(T_{l_B}V_B,d)}(l_B, l_B) \arrow{r} & \Si^{-1}\D\Ext^*_{(T_{l_A}V_A,d)}(l_A, l_A)
\end{tikzcd}
\]
is also a Lagrangian subspace and concentrated in degrees less than or equal to $\frac{3-d}{2}$. Thus, we may and will assume that we have $A=(T_{l_A}V_A,d)$, $B=(T_{l_B}V_B,d)$ and that the induced differentials on $V_A$, $V_B$ vanish. By Corollary~12.12 of \cite{VandenBergh15}, we may and will assume that we have $V_A=(\Si^{-1}\D\Ext^*_A(l_A,l_A))_{\leq 0}$ and $V_B=(\Si^{-1}\D\Ext^*_B(l_B,l_B))_{\leq 0}$.

\emph{Step 2. We decompose $V_A$ and $V_B$.}
The following proof is inspired by the ice quiver case in section~\ref{ss:Ginzburg-Lazaroiu morphisms} and we advise the reader to refer to that section constantly in order to follow the reasoning. By the implication from (1) to (3) in Theorem~10.4 of \cite{VandenBergh15}, the pseudocompact graded $l_B$-bimodule $V_B$ decomposes as $V^c_B\oplus z_B\, l_B$ with $z_B$ an $l_B$-central element of degree $2-d$, and the pseudocompact graded $l_B$-bimodule $V^c_B$ is of finite total dimension concentrated in degrees $[3-d,0]$. Note that the Calabi--Yau dimension of $B$ may be less than $3$. We can nevertheless apply the theorem because the proof does not use the assumption that the Calabi--Yau dimension is at least $3$. We define $R_B$ as the intersection of $V^c_B$ with the kernel of the induced map
\[
\Si^{-1}\D\Ext^*_B(l_B, l_B) \longrightarrow \Si^{-1}\D\Ext^*_A(l_A, l_A)\: .
\]
It is a pseudocompact graded $l_B$-subbimodule of $V^c_B$. Since the field $k$ is perfect and the $k$-algebra $l_B$ is semisimple, the $k$-algebra $l_B$ is separable and so is $l_B^e$. Thus, the pseudocompact graded $l_B$-bimodule $V^c_B$ decomposes as $F_B\oplus R_B$. By assumption, the graded $l_B$-bimodule $R_B$ is concentrated in degrees $[3-d,\frac{3-d}{2}]$ and the graded $l_B$-bimodule $F_B$ is concentrated in degrees $[\frac{3-d}{2},0]$. Since the morphism $f$ maps $F_B$ bijectively onto its image, by composing with its inverse we may and will assume that the restriction of $f$ to $F_B$ is the identity map. Since the kernel of the induced map
\[
\Si^{-1}\D\Ext^*_B(l_B, l_B) \longrightarrow \Si^{-1}\D\Ext^*_A(l_A, l_A)
\]
is a Lagrangian homogeneous graded subspace and the morphism $f$ induces the \linebreak bijection from $l_B\iso \D\Ext^0_B(l_B, l_B)$ onto its image, the kernel must contain the component $\Si^{-1}\D\Ext^{d-1}_B(l_B, l_B)\liso z_B\, l_B$. Thus, the element $f(z_B)$ is of tensor order at least $2$, \ie lies in $\prod_{p\geq 2}V^{\ten_{l_A} p}_A$. Since the induced differential on $V_A$ vanishes, the element $f(d(z_B))=d(f(z_B))$ is of tensor order at least $3$. But the images of nonzero elements of $F_B \ten_{l_B^e}F_B$ under $f$ are of tensor order $2$ (because the intersection of $F_B$ and the kernel of the induced map $\Si^{-1}\D\Ext^*_B(l_B, l_B) \to  \Si^{-1}\D\Ext^*_A(l_A, l_A)$ is zero) and the pseudocompact graded $l_B$-bimodule $R_B$ is isotropic with respect to the bilinear form on $\Si^{-1}\D\Ext^*_B(l_B, l_B)$, we deduce that the quadratic component $(\eta_B)_2=(d(z_B^\dagger))_2$ lies in $(F_B\ten_{l_B^e}R_B)\oplus (R_B\ten_{l_B^e}F_B)$. Note that this shows that $F_B$ is also a Lagrangian subspace and that $F_B$ and $R_B$ are in duality.

Since the morphism $f\colon B \to A$ is relative left $d$-Calabi--Yau, by the necessity in the relative version of Proposition~4.4.1 in \cite{KellerLiu23b} (\cf~section~4.2 of \cite{BravDyckerhoff19} for the relative version in the non-pseudocompact setting), the restriction dg functor $\res \colon \pvd\!_{dg}A \to \pvd\!_{dg}B$ is relative right $d$-Calabi--Yau. Therefore, we have the isomorphism
\[
\begin{tikzcd}
  \cocone(\res)\arrow{r}\arrow{d}{\wr}&_A(L,M)\arrow{r}{\res}\arrow{d}{\wr}&_B(L,M)\arrow{r}\arrow{d}{\wr}&\cone(\res)\arrow{d}{\wr}\\
  \Si^{-d}D_A(M,L)\arrow{r}&\Si^{-d}D\cocone(\res)\arrow{r}&\Si^{1-d}D_B(M,L)\arrow[swap]{r}&\Si^{1-d}D_A(M,L)
\end{tikzcd}
\]
of triangles, which is bifunctorial in $L$, $M\in \pvd A$. In this diagram, we write $_A(?,-)$ for $\Hom_{\pvd\!_{dg}A}(?,-)$ and similarly for $_B(?,-)$.
Since the pseudocompact dg $A$-module $l_A$ is finite-dimensional and its restriction $f_*(l_A)$ is isomorphic to $l_B$, if we let $L=M=l_A$ and take homology, we obtain the isomorphism
\begin{equation} \label{eq:diagram}
\adjustbox{max width=\textwidth}{
\begin{tikzcd}
_{B}(l_B,\Si^{p-1}l_B)\arrow{r}\arrow{d}{\wr}&H^p(\cocone(\res))\arrow{r}\arrow{d}{\wr}&_{A}(l_A,\Si^p l_A)\arrow{r}{\res}\arrow{d}{\wr}& _{B}(l_B,\Si^p l_B)\arrow{d}{\wr} \\
D_{B}(l_B,\Si^{d-p}l_B)\arrow[swap]{r}&D_{A}(l_A,\Si^{d-p}l_A)\arrow{r}&D H^{d-p}(\cocone(\res))\arrow{r}&D_{B}(l_B,\Si^{d-1-p}l_B)
\end{tikzcd}
}
\end{equation}
of long exact sequences. In this diagram, we write $_A(?,-)$ for $\Hom_{\pvd A}(?,-)$ and similarly for $_B(?,-)$. Since the dg algebra $A$ is connective, we have
\[
\begin{tikzcd}
\Ext^p_A(l_A,l_A) & \Ext^p_{H^0(A)}(l_A,l_A)=0 \arrow[swap]{l}{\sim}
\end{tikzcd}
\]
for all $p<0$ and similarly for $B$. By the diagram~(\ref{eq:diagram}), we have $\Ext^p_A(l_A,l_A)=0$ for all $p>d$. It follows that the graded $l_A$-bimodule $V_A$ is concentrated in degrees $[1-d,0]$. Since the dg $A$-module $l_A$ is finite-dimensional, by part~b) of Proposition~\ref{prop:finite dimension}, the graded vector space
\[
\Ext^*_A(l_A,l_A)\simeq \bigoplus_{p=0}^d \Hom_{\pvd A}(l_A,\Si^p l_A)
\]
is of finite total dimension and so is $V_A$.

The diagram~(\ref{eq:diagram}) yields an isomorphism
\begin{equation} \label{eq:duality}
\adjustbox{max width=\textwidth}{
\begin{tikzcd}
\ker(\res\colon \Ext^*_A(l_A,l_A) \to \Ext^*_B(l_B,l_B))\arrow{r}{\sim} & D\ker(\res\colon \Ext^{d-*}_A(l_A,l_A) \to \Ext^{d-*}_B(l_B,l_B)) \: .
\end{tikzcd}
}
\end{equation}
of graded $l_A$-bimodules. Since the pseudocompact dg algebra $B$ is connective and \linebreak $(d-1)$-Calabi--Yau, we have
\[
\begin{tikzcd}
\Ext^d_B(l_B,l_B)\arrow{r}{\sim} & D\Ext^{-1}_B(l_B,l_B)\arrow{r}{\sim} & D\Ext^{-1}_{H^0(B)}(l_B,l_B)=0 \: .
\end{tikzcd}
\]
If we let $*=0$ in the isomorphism (\ref{eq:duality}), we obtain the $l_A$-bimodule isomorphism
\[
\begin{tikzcd}
\ker(\res \colon \Hom_A(l_A,l_A) \to \Hom_B(l_B,l_B))\arrow{r}{\sim} & D\Ext^d_A(l_A,l_A)\: .
\end{tikzcd}
\]
The canonical projection from $l_A$ to $l_{\ol{A}}$ is an $l_{A}$-central generator of the $l_A$-bimodule on the left hand side and its annihilator is $l_B$. We define $sz_A$ as its image under the above $l_A$-bimodule isomorphism. Then the element $z_A$ is also an $l_A$-central generator of the pseudocompact graded $l_A$-bimodule $\Si^{-1}\D\Ext^d_A(l_A,l_A)$ which is of degree $1-d$ and its annihilator is $l_B$. We deduce that the pseudocompact graded $l_A$-bimodule $V_A$ decomposes as $V^c_A\oplus z_A\, l_{\ol{A}}$, where the pseudocompact graded $l_A$-bimodule $V^c_A$ is concentrated in degrees $[2-d,0]$.

Next we decompose $V_A$ further. We denote by $f_1 \colon V_B \to V_A$ the truncation of the \linebreak $(-1)$-shifted dual of the restriction map $\res \colon \Ext^*_A(l_A,l_A) \to \Ext^*_B(l_B,l_B)$. We claim that $\im f_1$ is an $l_A$-subbimodule of $V_A$. Clearly, it is stable under the actions of $l_B$ from both sides. Note that both the elements $1_A$ and $1_B$ act on $\im f_1$ by the identity from both sides, so the element $1_A-1_B$ annihilates $\im f_1$. This implies that the actions of $l_{\ol{A}}$ on $\im f_1$ from both sides are zero. Thus, the graded subspace $\im f_1$ is stable under the actions of $l_A$ from both sides, which means that it is an $l_A^e$-submodule of $V_A$. We define $F_A$ as $\im f_1$. We use the map $f_1$ to
identify $F_B$ with $F_A$. By degree reasons, the graded $l_A^e$-submodule $F_A$ is contained in $V^c_A$. Since the algebra $l_A^e$ is semisimple, the pseudocompact graded $l_A$-bimodule $V^c_A$ decomposes as $F_A\oplus N_A$, where the pseudocompact graded $l_A$-bimodule $N_A$ is concentrated in degrees $[2-d,0]$.

\emph{Step 3. We prove that $\eta_2$ and $(\eta_B)_2$ are non-degenerate.}
Since the element $z_A$ is \linebreak $l_A$-central and the element $z_B$ is $l_B$-central, we can write
\begin{align*}
&d(z_A)=\si_A'\, \eta_A\,\si_A''=\si_A'\, (\eta_A)_2\, \si_A''+\si_A'\, (\eta_A)_3\, \si_A''+\cdots \: , \\
&f(z_B)=\si_A'\, f(z_B)^\dag \si_A''=\si_A'\, f_2(z_B)^\dag \si_A''+\si_A'\, f_3(z_B)^\dag \si_A''+\cdots \: , \\
&d(z_B)=\si_B'\, \eta_B\,\si_B''=\si_B'\, (\eta_B)_2\, \si_B''+\si_B'\, (\eta_B)_3\, \si_B''+\cdots \: ,
\end{align*}
where $(\eta_A)_n$ and $f_n(z_B)^\dag$ are elements of $((V_A^c)^{\ten_{l_A}n})_{l_A}$ and $(\eta_B)_n$ is an element of \linebreak $((V_B^c)^{\ten_{l_B}n})_{l_B}$ for all $n\geq 2$. By the implication from (1) to (3) in Theorem~10.4 of \cite{VandenBergh15}, the element $(\eta_B)_2$ is non-degenerate in $V^c_B\ten_{l_B^e}V^c_B$. Note that the Calabi--Yau dimension of $B$ may be less than $3$. We can nevertheless apply the theorem because the proof does not use the assumption that the Calabi--Yau dimension is at least $3$. Put $\eta=\eta_A +f(z_B)^\dag$. Let us prove that the element $\eta_2$ is non-degenerate in $N_A\ten_{l^e_A}N_A$. The isomorphism
\[
\Ext^*_B(l_B,l_B)\xlongrightarrow{_\sim} D\Ext^{d-1-*}_B(l_B,l_B)
\]
of graded $l_B$-bimodules gives rise to a non-degenerate $l_B$-bilinear form
\[
\langle ?,-\rangle_B \colon \Ext^*_B(l_B,l_B)\ten_{l_B^e}\Ext^*_B(l_B,l_B) \longrightarrow k
\]
of degree $1-d$. The isomorphism (\ref{eq:duality}) of graded $l_A$-bimodules gives rise to a non-degenerate $l_A$-bilinear form
\[
\langle ?,-\rangle_A \colon \Ext_{A,B}^{*}(l_A,l_A)\ten_{l_A^e}\Ext_{A,B}^{*}(l_A,l_A) \longrightarrow k
\]
of degree $-d$, where $\Ext_{A,B}^{*}(?,-)$ is defined as the kernel of the restriction map
\[
\Ext^*_A(?,-) \longrightarrow \Ext^*_B(?,-) \: .
\]
We choose $A_\infty$-quasi-isomorphisms (\cf section~\ref{ss:Ainfty-algebras and Ainfty-modules} for a reminder on $A_\infty$-structures)
\[
\Ext^*_A(l_A,l_A) \longrightarrow \RHom_A(l_A,l_A) \quad \mbox{and} \quad \RHom_B(l_B,l_B) \longrightarrow \Ext^*_B(l_B, l_B)
\]
and define restriction $\res \colon \Ext^*_A(l_A,l_A) \to \Ext^*_B(l_B, l_B)$ as the composed 
$A_\infty$-algebra morphism
\[
\begin{tikzcd} 
\Ext^*_A(l_A,l_A) \arrow{r} &  \RHom_A(l_A,l_A) \arrow{r}{\res} & \RHom_B(l_B,l_B) \arrow{r} & \Ext^*_B(l_B, l_B)\: .
\end{tikzcd}
\]
We claim that we have 
\begin{equation} \label{eq:2nd component}
\langle g,h\rangle_A=(sz_A)(m_2(g,h))+(sz_B)(\res_2(g,h))
\end{equation}
for all $g$ and $h$ in $\Ext_{A,B}^{*}(l_A,l_A)$, where $\res_2$ denotes the second component of the \linebreak $A_\infty$-morphism $\res$.

We first consider the case that $g$ lies in $\Ext_{A,B}^{p}(l_A,l_{\ol{A}})$ and $h$ lies in $\Ext_{A,B}^{d-p}(l_{\ol{A}},l_A)$. By the functoriality of the isomorphism (\ref{eq:duality}), we obtain the commutative square
\[
\begin{tikzcd}
  \Ext_{A,B}^{0}(l_A,l_A)\arrow{r}{\sim}\arrow[swap]{d}{_{A}(g,l_A)}&D\Ext_{A,B}^{d}(l_A,l_A)\arrow{d}{D_{A}(l_A,g)}\\
  \Ext_{A,B}^{p}(l_A,l_A)\arrow{r}{\sim}&D\Ext_{A,B}^{d-p}(l_A,l_A)\mathrlap{\: .}
\end{tikzcd}
\]
By comparing the images of the canonical projection from $l_A$ to $l_{\ol{A}}$ under the two compositions in this commutative square we find that we have $\langle g,h\rangle_A=(sz_A)(m_2(g,h))$. 
Since $g$ lies in $\Ext^{p}_A(l_A,l_{\ol{A}})$ and $h$ lies in $\Ext^{d-p}_A(l_{\ol{A}},l_A)$, we have $\res_2(g,h)=0$. Therefore, the claimed equality~(\ref{eq:2nd component}) holds in this case.

We now consider the case that $g$ lies in $\Ext_{A,B}^{p}(l_A,l_B)$ and $h$ lies in $\Ext_{A,B}^{d-p}(l_A,l_B)$. By taking homology of the commutative square
\[
\begin{tikzcd}
  \RHom_A(l_A,l_A)\arrow{r}{\res}\arrow{d}{\wr}&\RHom_B(l_B,l_B)\arrow{d}{\wr}\\
  \Si^{-d}D\cocone(\res)\arrow{r}&\Si^{1-d}D\RHom_B(l_B,l_B)
\end{tikzcd}
\]
of dg modules over $\RHom_A(l_A,l_A)$ we obtain a commutative square
\[
\begin{tikzcd}
  \Ext^*_A(l_A,l_A)\arrow{r}{\res}\arrow{d}{\wr}&\Ext^*_B(l_B,l_B)\arrow{d}{\wr}\\
  D\Ext_{A,B}^{d-*}(l_A,l_A)\oplus D\Ext_{B,A}^{d-1-*}(l_A,l_A)\arrow[swap]{r}{[0~i]}&D\Ext^{d-1-*}_B(l_B,l_B)
\end{tikzcd}
\]
of $A_\infty$-modules over the $A_\infty$-algebra $\Ext^*_A(l_A,l_A)$, where $\Ext_{B,A}^{*}(?,-)$ is defined as the cokernel of the restriction map
\[
\Ext^*_A(?,-) \longrightarrow \Ext^*_B(?,-)\: .
\]
By considering the second component of the composed $A_\infty$-module morphism from the upper-left corner to the lower-right corner we find that we have
\[
\begin{bmatrix}
0 & i
\end{bmatrix}
_2(\langle g,?\rangle_A,h)+
\begin{bmatrix}
0 & i
\end{bmatrix}
_1(\langle ?,-\rangle_A)_2(g,h)=\langle \res_2(g,h),?\rangle_B+(\langle ?,-\rangle_B)_2(\res_1(g),h) \: .
\]
Since $g$ lies in $\Ext_{A,B}^p(l_A,l_B)$ and the vector space $\Ext_{B,A}^0(l_A,l_A)$ vanishes, it reduces to $
\begin{bmatrix}
0 & i
\end{bmatrix}
_2(\langle g,?\rangle_A,h)=\langle \res_2(g,h),?\rangle_B$. We calculate the left hand side as follows. Denote by $M$ the $A_\infty$-module
\[
\cocone(D\res \colon D\Ext^{d-1-*}_B(l_B,l_B) \to D\Ext^{d-1-*}_A(l_A,l_A))\: .
\]
As a graded vector space, it decomposes (non-canonically) as
\[
D\ol{\Ext}_{A,B}^{d-*}(l_A,l_A)\oplus D\Ext_{A,B}^{d-*}(l_A,l_A)\oplus D\Ext_{B,A}^{d-1-*}(l_A,l_A)\oplus D\ol{\Ext}_{A,B}^{d-1-*}(l_A,l_A)\: ,
\]
where $\ol{\Ext}_{A,B}^*(?,-)$ is defined as the image of the restriction map
\[
\Ext^*_A(?,-) \longrightarrow \Ext^*_B(?,-)\: .
\]
Let $I$ be the canonical injection
\[
\begin{tikzcd}
D\Ext_{A,B}^{d-*}(l_A,l_A)\oplus D\Ext_{B,A}^{d-1-*}(l_A,l_A)\arrow{r} & M
\end{tikzcd}
\]
and $P$ the canonical projection
\[
\begin{tikzcd}
M\arrow{r} & D\Ext_{A,B}^{d-*}(l_A,l_A)\oplus D\Ext_{B,A}^{d-1-*}(l_A,l_A)
\end{tikzcd}
\]
and $H$ the composition of the canonical maps
\[
\begin{tikzcd}
M \arrow{r} & D\Ext^{d-*}_A(l_A,l_A) \arrow{r} & D\ol{\Ext}_{A,B}^{d-*}(l_A,l_A) \arrow{dl} & \\
& D\ol{\Ext}_{A,B}^{d-1-(*-1)}(l_A,l_A) \arrow{r} & D\Ext^{d-1-(*-1)}_B(l_B,l_B) \arrow{r} & M\: ,
\end{tikzcd}
\]
where the third map is of degree $-1$. They satisfy
\[
P\circ I=\id_{H^*(M)}\ko \id_M-I\circ P=d(H)\ko H\circ I=0\ko P\circ H=0\ko H^2=0\: .
\]
This means that $(I,P,H)$ exhibits $H^*(M)$ as a deformation retract of $M$, \cf section~1.1 of \cite{Vallette14}. The morphism
\[
\begin{bmatrix}
0 & i
\end{bmatrix}
\colon
\begin{tikzcd}
D\Ext_{A,B}^{d-*}(l_A,l_A)\oplus D\Ext_{B,A}^{d-1-*}(l_A,l_A) \arrow{r} & D\Ext^{d-1-*}_B(l_B,l_B)
\end{tikzcd}
\]
of $A_\infty$-modules is the composition
\[
\begin{tikzcd}[ampersand replacement=\&]
H^*(M)\arrow{rr}{
\begin{bsmallmatrix}
0 & 0 \\
\id & 0 \\
0 & \id \\
0 & 0
\end{bsmallmatrix}
} \& \& M \arrow{rr}{[0~0~i~i]} \& \& D\Ext^{d-1-*}_B(l_B,l_B)\: ,
\end{tikzcd}
\]
where the second morphism is strict. By the variant of Proposition~7 of \cite{KontsevichSoibelman01} for \linebreak $A_\infty$-modules, \cf also Theorem~5 of \cite{Vallette14}, it follows that we have
\begin{align*}
&\begin{bmatrix} 0 & i \end{bmatrix}_2(\langle g,-\rangle_A,h) \\
=&\begin{bmatrix} 0 & 0 & i & i \end{bmatrix}_1(\begin{bmatrix} 0 & 0 \\ \id & 0 \\ 0 & \id \\ 0 & 0 \end{bmatrix}_2(\langle g,-\rangle_A,h)) \\
=&\begin{bmatrix} 0 & 0 & i & i \end{bmatrix}_1(H(m_2(I(\langle g,-\rangle_A),h))) \\
=&\begin{bmatrix} 0 & 0 & i & i \end{bmatrix}_1(H(m_2(\langle g,-\rangle_A,h))) \\
=&\begin{bmatrix} 0 & 0 & i & i \end{bmatrix}_1(H(\langle g,h\circ-\rangle_A))\: .
\end{align*}
Since $g$ lies in $\Ext^p_A(l_A,l_B)$ and $h$ lies in $\Ext^{d-p}_A(l_B,l_A)$, the element $\langle g,h\circ-\rangle_A$ actually lies in the graded subspace $D\Ext_{B,A}^{d-*}(l_A,l_A)$. Thus, we obtain
\[
\begin{bmatrix}
0 & 0 & i & i
\end{bmatrix}
_1(H(\langle g,h\circ-\rangle_A))=
\begin{bmatrix}
0 & 0 & i & i
\end{bmatrix}
_1(\langle g,h\circ-\rangle_A)=\langle g,h\circ-\rangle_A\: .
\]
We conclude that we have $\langle g,h\circ-\rangle_A = \langle \res_2(g,h),-\rangle_B$. If we evaluate this identity
at $\id_{l_B}$, we deduce that we have $\langle g,h\rangle_A=(sz_B)(\res_2(g,h))$.
Since $g$ lies in $\Ext^p_A(l_A,l_B)$ and $h$ lies in $\Ext^{d-p}_A(l_B,l_A)$, the element $m_2(g,h)$ lies in $\Ext^d_A(l_B,l_B)$. But the $l_A$-bimodule $D\Ext^d_A(l_B,l_B)$ is isomorphic to $\Ext_{A,B}^0(l_B,l_B)=0$, so we have $(sz_A)(m_2(g,h))=0$. Therefore, the claimed equality~(\ref{eq:2nd component}) also holds in this case.

Finally, since the graded vector space $\Ext_{A,B}^*(l_A,l_A)\ten_{l_A^e}\Ext_{A,B}^{d-*}(l_A,l_A)$
decomposes as
\[
(\Ext_{A,B}^*(l_A,l_{\ol{A}})\ten_{l_A^e}\Ext_{A,B}^{d-*}(l_{\ol{A}},l_A))\oplus(\Ext_{A,B}^*(l_A,l_B)\ten_{l_A^e}\Ext_{A,B}^{d-*}(l_B,l_A))\: ,
\]
the general case can be reduced to the two cases above. In conclusion, the claimed
equality~(\ref{eq:2nd component}) holds in general. By taking the $(-1)$-shifted dual 
we deduce that the image of $1$ under the composed map
\[
\begin{tikzcd}
k\arrow{r}&\Si^{-1}\D\Ext^d_A(l_A,l_A)\oplus \Si^{-1}\D\Ext^{d-1}_B(l_B,l_B)
\arrow{d}{[d_2~f_2]} \\
 & \bigoplus_{p=1}^{d-1}(\Si^{-1}\D\Ext_{A,B}^p(l_A,l_A)\ten_{l_A^e}\Si^{-1}\D\Ext_{A,B}^{d-p}(l_A,l_A))
\end{tikzcd}
\]
is non-degenerate. Now $\eta_2$ was defined as this image. 

\emph{Step 4. We reduce to the case where $\eta$ and $\eta_B$ are sums of graded commutators.}
Recall that the exact relative $d$-Calabi--Yau structure on the morphism $f\colon B \to A$ is given by a class
in the relative cyclic homology $\HC_{d-1}(A,B)$.
Using the relative version of the description of reduced cyclic homology given in Proposition~\ref{prop:cyclic homology} 
we choose a representative $(\chi_A,s\chi_B)$ of the underlying reduced cyclic class of this class,  where 
\begin{itemize}
\item[a)] the element $\chi_B$ of $(B/l_B)_{l_B}$ is of degree $2-d$ and we have 
\[
\ol{d(\chi_B)}=0 \mbox{ in } (B/(l_B+[B,B]))_{l_B}\: ,
\]
\item[b)] the element $\chi_A$ of $(A/l_A)_{l_A}$ is of degree $1-d$ and we have 
\[
\ol{d(\chi_A)+f(\chi_B)}=0 \mbox{ in } (A/(l_A+[A,A]))_{l_A}\: .
\]
\end{itemize}
In other words, we have
\[
d(\chi_B)=\sum_j[x^B_j,y^B_j]\quad \mod [l_B,?]
\]
for suitable $x^B_j$ and $y^B_j$ in $B/l_B$ and 
\[
d(\chi_A)+f(\chi_B)=\sum_i[x^A_i,y^A_i]\quad \mod [l_A,?]
\]
for suitable $x^A_i$ and $y^A_i$ in $A/l_A$. By the relative version of Proposition~\ref{prop:connecting morphism}, the class $[(\ol{\chi_A},s\ol{\chi_B})]$ is mapped to $[((0,-sD\chi_A),s(0,sD\chi_B))]$ by Connes' map $B$. Since the relative Hochschild class $[((0,-sD\chi_A),s(0,sD\chi_B))]$ in $\HH_d(A,B)$ is non-degenerate, the Hochschild class $[(0,sD\chi_B)]$ in $\HH_{d-1}(B)$ is non-degenerate (and thus gives rise to a \linebreak $(d-1)$-Calabi--Yau structure on $B$). By the necessity in Lemma~10.2 of \cite{VandenBergh15}, the element $\chi_B$ is of the form $u_Bz_B^\dagger+v_B$ for an invertible central element  $u_B$ of $l_B$ and an element  $v_B$ of $B$ of tensor order at least $2$. The homotopy cofibre $\ol{A}$ of $f\colon B\to A$ is isomorphic to $(T_{l_{\ol{A}}}(l_{\ol{A}}\ten_{l_A}V_A\ten_{l_A}l_{\ol{A}}),d)$, where the graded $l_{\ol{A}}$-bimodule $l_{\ol{A}}\ten_{l_A}V_A\ten_{l_A}l_{\ol{A}}$ is concentrated in degrees $[1-d,0]$ and its component of degree $1-d$ is
\[
\begin{tikzcd}
l_{\ol{A}}\ten_{l_A}\Si^{-1}\D\Ext^d_A(l_A,l_A)\ten_{l_A}l_{\ol{A}}\arrow{r}{\sim} & \Si^{-1}\D\Ext^d_A(l_A,l_A)=z_A\, l_{\ol{A}},
\end{tikzcd}
\]
where $z_A$ is an $l_{\ol{A}}$-central element. By the variant of Corollary~7.1 of \cite{BravDyckerhoff19} for pseudocompact dg algebras, the homotopy cofibre is $d$-Calabi--Yau and, more precisely, the image of $[(0,-sD\chi_A)]$ in
\[
\HH_{d}(\ol{A})=H^{-d}(\cone(\del_1 \colon (\Omega^1_{l_{\ol{A}}}\ol{A})_\natural \to \ol{A}_{l_{\ol{A}}}))
\]
is non-degenerate. By the necessity in Lemma~10.2 of \cite{VandenBergh15} again, the image of $\chi_A$ in $(\ol{A}/l_{\ol{A}})_{l_{\ol{A}}}$ is of the form $u_Az_A^\dagger+v_{\ol{A}}$ for an invertible central element $u_A$ of $l_{\ol{A}}$ and an element $v_{\ol{A}}$ of $\ol{A}$ of tensor order at least $2$. Therefore, the element $\chi_A$ is of the form $u_Az_A^\dagger+v_A$, where $v_A$ is an element of $A$ of tensor order at least $2$. Put 
\[
z_A'=\si_A'\, \chi_A\, \si_A''\ko z_B'=\si_B'\, \chi_B\, \si_B'' \ko V_A'=V^c_A\oplus z_A'\, l_{\ol{A}}\mbox{ and }V_B'=V^c_B\oplus z_B'\, l_B \: .
\]
We have $d(z_B')=\si_B'\, d(\chi_B)\si_B''$ such that $d(\chi_B)$ is a sum of graded commutators in $T_{l_B}V^c_B$ and
\[
(d(\chi_B))_2=(d(u_B\, z_B^\dagger+v_B))_2=u_B(\eta_B)_2
\]
is an element of $(F_B\ten_{l_B^e}R_B)\oplus (R_B\ten_{l_B^e}F_B)$ whose image in $V^c_B\ten_{l_B^e}V^c_B$ is non-degenerate. We also have $d(z_A')+f(z_B')=\si_A'(d(\chi_A)+f(\chi_B))\si_A''$ such that $d(\chi_A)+f(\chi_B)$ is a sum of graded commutators in $T_{l_A}V^c_A$. The element $d(v_A)$ of $A$ is of tensor order at least $3$ because $d_1$ vanishes. The element $f(v_B)$ of $A$ is also of tensor order at least $3$ because the quadratic terms of $v_B$ must contain a tensor factor in $R_B$ since they are of degree $2-d$. So the element
\[
(d(\chi_A)+f(\chi_B))_2=(d(u_A\, z_A^\dagger+v_A)+f(u_B\, z_B^\dagger+v_B))_2=(u_A+u_B)\eta_2
\]
is also non-degenerate in $N_A\ten_{l_A^e}N_A$. Let $q_A\colon (T_{l_A}V_A,d) \to (T_{l_A}V_A',d)$ be the morphism of $l_A$-augmented pseudocompact dg algebras which restricts to the identity on $V^c_A$ and maps $z_A$ to $z_A'$. Since the element $u_A$ is invertible in $l_{\ol{A}}$, the induced morphism $(q_A)_1\colon V_A \to V_A'$ of $l_A$-bimodules is an isomorphism. This means that $q_A$ is an isomorphism. Hence we have $T_{l_A}V_A=T_{l_A}V_A'$ and similarly for $B$.
After replacing $z_A$, $z_B$, $V_A$, $V_B$ by $z_A'$, $z_B'$, $V_A'$, $V_B'$, respectively, we may and will assume that $\eta$ is a sum of graded commutators in $T_{l_A}V^c_A$ and that $\eta_B$ is a sum of graded commutators in $T_{l_B}V^c_B$.

\emph{Step 5. We remove the higher terms from $\eta$ and $\eta_B$.}
By the implication from (3) to (2) in Theorem~10.4 of \cite{VandenBergh15}, 
we can remove all the terms of tensor degree at least $3$ from $\eta_B$. Note that the Calabi--Yau dimension of $B$ may be less than $3$. We can nevertheless apply the theorem because
the proof does not use the assumption that the Calabi--Yau dimension is at least $3$. Thus, we may and will assume that we have $(\eta_B)_n=0$ for all $n\geq 3$.
We will now remove all the terms of tensor degree at least $3$ from $\eta$. Since $\eta$ is a sum of graded commutators, it lies in the sum $([F_A,T_{l_A}V^c_A]+[N_A,T_{l_A}V^c_A])_{l_A}$. In the spirit of the proof of the implication from (3) to (2) in Theorem~10.4 of \cite{VandenBergh15}, we will first use induction to remove the second summand of a chosen sum decomposition of $\eta$. Assume that we have shown that the $\varphi$-augmented morphism $f\colon B\to A$ between pseudocompact dg algebras is weakly equivalent to one such that $\eta_3$, \ldots, $\eta_{n-1}$ lie in $[F_A,T_{l_A}V^c_A]_{l_A}$ for some $n\geq 3$. We will construct an isomorphism $q\colon (T_{l_A}V_A,d) \to (T_{l_A}V_A,d')$ of $l_A$-augmented pseudocompact dg algebras which is determined by $q(v)=v+\beta(v)$ for $v$ in $V_A$ such that the components of tensor degrees $[3, n]$ of the element $d'(z_A)+(q\circ f)(z_B)$ lie in $[F_A,T_{l_A}V^c_A]$, where $\beta$ is an $l^e_A$-linear map from $V_A$ to $(V^c_A)^{\ten_{l_A}n-1}$ which vanishes on $F_A \oplus z_A\, l_{\ol{A}}$. Then we have $d'=q\circ d\circ q^{-1}$. This implies the equalities
\begin{align*}
d'(z_A)+(q\circ f)(z_B) &=q(d(z_A)+f(z_B)) \\
                        &=\si_A'\, q(\eta)\si_A'' \\
                        &=\si_A'\, \eta_2\, \si_A''+\si_A'\, \eta_n\, \si_A''+\si_A'\, \beta(\eta_2')\eta_2''\, \si_A''+\si_A'\, \eta_2'\beta(\eta_2'')\, \si_A''+\cdots \: ,
\end{align*}
where the omitted terms lie in $\prod_{p\geq n+1}(V_A^c)^{\ten_{l_A} p}+[F_A,T_{l_A}V^c_A]$.
It suffices to find an \linebreak $l^e_A$-linear map $\beta$ satisfying $\eta_n+\beta(\eta_2')\eta_2''+\eta_2'\beta(\eta_2'')\in [F_A,T_{l_A}V^c_A]_{l_A}$. Since $\eta$ is a sum of graded commutators, its quadratic component $\eta_2$ is graded anti-symmetric. If we apply $\beta \ten \id_{N_A}$ to $\eta_2'\eta_2''=-(-1)^{|\eta_2'||\eta_2''|}\eta_2''\eta_2'$, the condition can be written as 
\[
\eta_n+[\eta_2',\beta(\eta_2'')]\in [F_A,T_{l_A}V^c_A]_{l_A}\: .
\]
Since this is a linear algebra problem, we may and will assume that the field $k$ is algebraically closed. It is also invariant under Morita equivalences, so we may and will assume that the $k$-algebra $l_A$ equals $\prod_{i=1}^m ke_i$. Since the element $\eta_2$ is non-degenerate and graded anti-symmetric in $N_A\ten_{l_A^e}N_A$, if we choose a suitable homogeneous $k$-basis $\ca$ of $N_A$ endowed with an involution $*$ which maps $a$ to $a^*$ (note that the involution $*$ may have fixed points), we can write the element $\eta_2$ as the sum $\sum_{a\in \ca/*}[a,a^*]$. So we have
\[
[\eta_2',\beta(\eta_2'')]=\sum_{a\in \ca/*}([a,\beta(a^*)]-(-1)^{|a||a^*|}[a^*,\beta(a)])\: .
\]
The component $\eta_n$ can be written as
\[
\eta_n=\sum_{a\in \ca/*}([a,\eta_a]+[a^*,\eta_{a^*}])+\cdots \quad \mod [l_A,?]
\]
for certain linear combinations $\eta_a$ and $\eta_{a^*}$ of paths of length $n-1$, where we let $\eta_a=\eta_{a^*}$ if $a$ equals $a^*$ and we omit the terms in $[F_A,T_{l_A}V^c_A]_{l_A}$. Now it suffices to put \mbox{$\beta(a)=(-1)^{|a||a^*|}\eta_{a^*}$} and $\beta(a^*)=-\eta_a$. Note that if $a$ equals $a^*$, then $|a|=|a^*|$ is odd and hence $-(-1)^{|a||a^*|}$ equals $1$, as it should. For degree reasons, the graded subalgebra $T_{l_A}V^c_A$ is stable under the morphism $q$. This implies that the element $q(\eta)$ is also a sum of graded commutators in $(T_{l_A}V^c_A)_{l_A}$ and its quadratic component $q(\eta)_2=\eta_2$ is also non-degenerate in $N_A\ten_{l_A^e}N_A$. Therefore, we have shown that the $\varphi$-augmented morphism $f$ between pseudocompact dg algebras is
isomorphic (in the model category $\PCAlgc \varphi$) to $q\circ f$ such that $\eta_3$, \ldots, $\eta_n$ lie in $[F_A,T_{l_A}V^c_A]_{l_A}$ and the properties in the previous steps are preserved. Since the dg algebra $A$ is pseudocompact, by taking the limit of this procedure we may and will assume that we have $\eta_n \in [F_A,T_{l_A}V^c_A]_{l_A}$ for all $n\geq 3$. Next, we will replace $f$ with a homotopic morphism to remove $\eta_n$ for all $n\geq 3$. Since the element $\eta_B$ is non-degenerate and graded anti-symmetric in $(F_B\oplus R_B)\ten_{l^e_B}(F_B\oplus R_B)$, by choosing a suitable homogeneous $k$-basis $\cb$ of $F_B$ we can write the element $\eta_B$ as the sum $\sum_{b\in \cb}[b,b^*]$. Here the elements $b^*$ form the homogeneous $k$-basis of $R_B$ which is $k$-dual to $\cb$ with respect to $\eta_B$. Because now the element $\sum_{n\geq 3}\eta_n$ lies in $[F_A,T_{l_A}V^c_A]_{l_A}$, it can be written as $\sum_{b\in \cb}[b,\eta_b]$, where the elements $\eta_b$ are of tensor order at least $2$. To remove the terms of tensor degree at least $3$ from $\eta$, we will construct a continuous $l^e_B$-linear map $h\colon B \to A$ of degree $-1$ which vanishes on $l_B$ and satisfies
\[
h(b_1b_2)=h(b_1)(f(b_2)+d(h(b_2))+h(d(b_2)))+(-1)^{|b_1|}f(b_1)h(b_2)
\]
for all $b_1$ and $b_2$ in $B$. Using double induction on the pair formed by the internal degree and the tensor degree we see that $h$ is determined by its restriction to $V_B$, which can be chosen arbitrarily.
Put $g=f+d\circ h+h\circ d$. It follows from the construction that the morphism $g$ is also a $\varphi$-augmented morphism between pseudocompact dg algebras
and that the map $h$ is an $f$-$g$-derivation of degree $-1$. Since we have
\[
d(z_A)+g(z_B)=d(z_A)+f(z_B)+d(h(z_B))+h(d(z_B))\: ,
\]
it suffices to find a map $h$ satisfying $h(z_B)=0$ and $\sum_{b\in \cb}[b,\eta_b]+h(\eta_B)=0$ in $(T_{l_A}V^c_A)_{l_A}$. Since this is a linear algebra problem, we may and will assume that the field $k$ is algebraically closed. It is also invariant under Morita equivalences, so we may and will assume that the $k$-algebra $l_A$ equals $\prod_{i=1}^m ke_i$. Then the above equation can be written as
\begin{align*}
 & \sum_{b\in \cb}[b,\eta_b]+\sum_{b\in \cb}(h(b)(f(b^*)+d(h(b^*))+h(d(b^*)))+(-1)^{|b|}f(b)h(b^*)\\
- & (-1)^{|b||b^*|}(h(b^*)(f(b)+d(h(b))+h(d(b)))+(-1)^{|b^*|}f(b^*)h(b)))=0\: .
\end{align*}
Note that for degree reasons, the differential $d(b)$ must lie in $T_{l_B}F_B$ if $b$ lies in $F_B$. So the $l^e_B$-linear map $h$ which vanishes on $F_B\oplus z_B\, l_B$ and satisfies $h(b^*)=-(-1)^{|b|}\eta_b$ is a solution to this equation. Therefore, the element $(d(z_A)+g(z_B))^\dagger$ does not have terms of tensor degree at least $3$ and its quadratic component equals $\eta_2$. By part~a) of Proposition~1.3.4.1 of \cite{Lefevre03} (translated from cocomplete augmented dg coalgebras to complete augmented pseudocompact dg algebras), part~(ii) of Lemma~4.21 and the sufficiency in Lemma~4.24 of \cite{DwyerSpalinski95}, the $\varphi$-augmented morphisms $f$ and $g$ between pseudocompact dg algebras are weakly equivalent in the model category $\PCAlgc \varphi$. By replacing $f$ by $g$ we may and will assume that we have $\eta_n=0$ for all $n\geq 3$.

\emph{Step 6. We describe the differentials of $V_A^c$, $V_B^c$ and the restriction of the morphism $f$ to $R_B$.}
We now prove there are elements $w_A \in \Tr(T_{l_A}V^c_A)$ and $w_B \in \Tr(T_{l_B}V^c_B)$ satisfying the conditions in statement i) of Theorem~\ref{thm:main}. If we have $d<4$, then for degree reasons, we have to put $w_B=0$. If we have $d\geq 4$, then, by Lemma~10.5 of \cite{VandenBergh15}, there is an element $w_B \in \Tr(T_{l_B}V^c_B)$ such that we have $d(v)=\{w_B,v\}_{\omega_{\eta_B}}$ for all $v\in F_B \oplus R_B$. For $v$ in $F_A$, since we have assumed that $f$ restricts to the identity $F_B \to F_A$, we have
\[
d(v)=d(f(v))=f(d(v))=f(\{w_B,v\}_{\omega_{\eta_B}})=\{w_B,v\}_{\omega_{\eta_B}}\: .
\]
Now, in the spirit of the proof of Lemma~10.5 of \cite{VandenBergh15}, we consider the differential of $N_A$. We have
\[
d(d(z_A)+f(z_B))=d(f(z_B))=f(d(z_B))\: .
\]
Since we have $\eta_B=(d(z_B))^\dagger$ and $\eta=(d(z_A)+f(z_B))^\dagger$, this means that we have
\[
d(\eta')\eta''+(-1)^{|\eta'|}\eta'd(\eta'')=f(\eta_B)\quad \mod [l_A,?]\: .
\]
If we apply $d\ten \id_{N_A}$ to $\eta'\ten \eta''=-(-1)^{|\eta'||\eta''|}\eta''\ten \eta'$, this equality can be written as
\[
(-1)^{|\eta'|}\eta'd(\eta'')-(-1)^{|\eta'||\eta''|}d(\eta'')\eta'=f(\eta_B)\quad \mod [l_A,?]\: .
\]
As before, we write the element $\eta_B$ as the sum $\sum_{b\in \cb}[b,b^*]$. Hence we have
\[
(-1)^{|\eta'|+1}\eta'd(\eta'')+\sum_{b\in \cb}bf(b^*)=(-1)^{|\eta'||\eta''|+1}d(\eta'')\eta'+\sum_{b\in \cb}(-1)^{|b||b^*|}f(b^*)b\quad \mod [l_A,?]\: .
\]
Consequently, the component of each tensor degree of $(-1)^{|\eta'|+1}\eta'd(\eta'')+\sum_{b\in \cb}bf(b^*)$ is stable under the generator of the corresponding cyclic permutation group. 
This means that
\[
\ol{w_A}=(-1)^{|\eta'|+1}\eta'd(\eta'')+\sum_{b\in \cb}bf(b^*)
\]
is a (componentwise) cyclically symmetric element of $(T_{l_A}V^c_A)_{l_A}$ which is of degree $3-d$. Let $w_A$ be a preimage of $\ol{w_A}$ (recall that our ground field $k$ is of characteristic $0$) under the cyclic symmetrisation map
\[
\sym \colon \Tr(T_{l_A}V^c_A)\longrightarrow (\prod_{p\geq 1}(V_A^c)^{\ten_{l_A} p})_{l_A} \: .
\]
Then the element $w_A$ only contains cubic and higher terms. For a morphism $\phi \colon N_A\to l_A^e$ of $l_A^e$-modules, we define the map $\del_{\phi}\colon (T_{l_A}V^c_A)_{l_A} \to T_{l_A}V^c_A$ which maps $a_1\ten \cdots \ten a_n$ to $\phi(a_1)''a_2\ten \cdots \ten a_n\phi(a_1)'$,
where we extend $\phi$ by zero from $N_A$ to $V_A^c=F_A\oplus N_A$. The element $\eta$ of $N_A\ten_{l_A^e}N_A$, which is of degree $2-d$, defines a morphism $\eta^+ \colon \Hom_{l_A^e}(N_A,l_A^e) \to N_A$ of degree $2-d$ which maps $\phi$ to $(-1)^{|\phi||\eta|}\phi(\eta')''\eta''\phi(\eta')'$. So we have
\[
d(\eta^+(\phi))=(-1)^{|\phi||\eta|}\phi(\eta')''d(\eta'')\phi(\eta')'=(-1)^{|\phi||\eta|}(-1)^{|\phi|+1}\del_{\phi}(\ol{w_A})\: .
\]
For a morphism $\phi \colon N_A\to l_A^e$ of $l_A^e$-modules, we have the associated double $l_A$-derivation $i_\phi \colon T_{l_A}V^c_A \to T_{l_A}V^c_A\ten_k T_{l_A}V^c_A$ which maps $v$ to $\phi(v)$. We define the induced map \linebreak $\iota_\phi \colon \Tr(T_{l_A}V^c_A) \to T_{l_A}V^c_A$ which maps $\ol{f}$ to $(-1)^{|i_\phi(f)''||i_\phi(f)'|}i_\phi(f)''i_\phi(f)'$. In particular, it maps $a_1\ten \cdots \ten a_n$ to
\[
\sum_i\pm \phi(a_i)''a_{i+1}\ten \cdots \ten a_n \ten a_1\ten \cdots \ten a_{i-1}\phi(a_i)'\: .
\]
From $\iota_\phi(w_A)=\del_\phi(\ol{w_A})$, we deduce that we have $d(\eta^+(\phi))=(-1)^{|\phi||\eta|}(-1)^{|\phi|+1}\iota_{\phi}(w_A)$. Since the element $\eta$ is non-degenerate in $N_A\ten_{l_A^e}N_A$, the morphism $\eta^+$ has an inverse $\eta^- \colon N_A \to \Hom_{l_A^e}(N_A,l_A^e)$ of degree $d-2$. If we apply $\phi=\eta^-(v)$ to any homogeneous element $v$ of $N_A$, we obtain
\[
d(v)=(-1)^{(|v|+d-2)(d-2)+(|v|+d-2)+1}\iota_{\eta^-(v)}(w_A)=(-1)^{|v|(d+1)+1}\iota_{\eta^-(v)}(w_A) \: ,
\]
where we have
\[
v=\eta^+(\phi)=(-1)^{|\phi||\eta|}\phi(\eta')''\eta''\phi(\eta')'=-(-1)^{|\phi||\eta|}(-1)^{|\eta'||\eta''|}\phi(\eta'')''\eta'\phi(\eta'')'\: .
\]
Since the restriction of $D$ to $V_A^c$ is injective, we may identify $V_A^c$ with its image under $D$ to write
\begin{align*}
2\iota_\phi(\omega_\eta) &=\iota_\phi ((D\eta')(D\eta'')) \\
                         &=\phi(\eta')''(D\eta'')\phi(\eta')'-(-1)^{|\eta'||\phi|}\phi(\eta'')''(D\eta')\phi(\eta'')' \\
                         &=2(-1)^{|\phi||\eta|}Dv\: .
\end{align*}
Finally, we find that we have
\[
\iota_{\eta^-(v)}(\omega_\eta)=(-1)^{(|v|+d-2)(d-2)}Dv=(-1)^{d(|v|+1)}Dv\: .
\]
Since the element $\eta$ is non-degenerate and graded anti-symmetric in $N_A\ten_{l_A^e}N_A$, it yields a double Poisson bracket $\ldb?,-\rdb_{\omega_\eta}$ on $T_{l_A}V^c_A$, cf.~section~\ref{ss:the necklace bracket}. So we have
\[
\ldb u,v\rdb_{\omega_\eta}=-(-1)^{(|u|-|\omega_\eta|)(|v|-|\omega_\eta|)}(-1)^{|\ldb u,v\rdb_{\omega_\eta}'||\ldb u,v\rdb_{\omega_\eta}''|}\ldb v,u\rdb_{\omega_\eta}''\ten \ldb v,u\rdb_{\omega_\eta}'\: .
\]
In particular, we have $|\ldb u,v\rdb_{\omega_\eta}'|=|\ldb v,u\rdb_{\omega_\eta}''|$ and $|\ldb u,v\rdb_{\omega_\eta}''|=|\ldb v,u\rdb_{\omega_\eta}'|$. Therefore, we have
\begin{align*}
d(v) &=(-1)^{|v|+d+1}\iota_{H_v}(w_A) \\
     &=(-1)^{|v|+d+1}\iota_{H_v}(Dw_A) \\
     &=(-1)^{|v|+d+1}(-1)^{|i_{H_v}(Dw_A)''||i_{H_v}(Dw_A)'|}i_{H_v}(Dw_A)''i_{H_v}(Dw_A)' \\
     &=(-1)^{|v|+d+1}(-1)^{|H_v(w_A)''||H_v(w_A)'|}\ldb v,w_A\rdb_{\omega_\eta}''\ldb v,w_A\rdb_{\omega_\eta}' \\
     &=\pm \ldb w_A,v\rdb_{\omega_\eta}'\ldb w_A,v\rdb_{\omega_\eta}''\: ,
\end{align*}
where the sign is given by the parity of
\[
1+(|v|+d+1)+|\ldb v,w_A\rdb_{\omega_\eta}''||\ldb v,w_A\rdb_{\omega_\eta}'|+(|v|+d-2)(|w_A|+d-2)+|\ldb w_A,v\rdb_{\omega_\eta}'||\ldb w_A,v\rdb_{\omega_\eta}''|\: .
\]
Thus, we have
\begin{align*}
d(v) &=-(-1)^{|v|+d+1}(-1)^{(|v|+d-2)(|w_A|+d-2)}\ldb w_A,v\rdb_{\omega_\eta}'\ldb w_A,v\rdb_{\omega_\eta}'' \\
     &=-(-1)^{|v|+d+1}(-1)^{|v|+d}\{w_A,v\}_{\omega_\eta} \\
     &=\{w_A,v\}_{\omega_\eta}
\end{align*}
for all homogeneous elements $v$ of $N_A$. Since both sides are additive in $v$, 
the same equality holds for all $v$ in $N_A$.

Next, in the spirit of the proof of Lemma~10.5 of \cite{VandenBergh15}, we consider the image of $R_B$ under $f$. For a morphism $\phi \colon F_B\to l_B^e$ of $l_B^e$-modules, we define the map $\del_{\phi}\colon (T_{l_A}V^c_A)_{l_A} \to T_{l_A}V^c_A$ which maps $a_1\ten \cdots \ten a_n$ to $\phi(a_1)''a_2\ten \cdots \ten a_n\phi(a_1)'$, where we extend $\phi$ by zero from $F_A$ to $V_A^c=F_A\oplus N_A$. The element $\ol{\eta_B}=\sum_{b\in \cb}bb^*$ of $F_B\ten_{l_B^e}R_B$, which is of degree $3-d$, defines a morphism $\ol{\eta_B}^+ \colon \Hom_{l_B^e}(F_B,l_B^e) \to R_B$ of degree $3-d$ which maps $\phi$ to $(-1)^{|\phi||\ol{\eta_B}|}\phi(\ol{\eta_B}')''\ol{\eta_B}''\phi(\ol{\eta_B}')'$. So we have
\[
f(\ol{\eta_B}^+(\phi))=(-1)^{|\phi||\ol{\eta_B}|}\phi(\ol{\eta_B}')''f(\ol{\eta_B}'')\phi(\ol{\eta_B}')'=(-1)^{|\phi||\ol{\eta_B}|}\del_{\phi}(\ol{w_A})\: .
\]
For a morphism $\phi \colon F_B\to l_B^e$ of $l_B^e$-modules, we have the associated double $l_A$-derivation $i_\phi \colon T_{l_A}V^c_A \to T_{l_A}V^c_A\ten_k T_{l_A}V^c_A$ which maps $v$ to $\phi(v)$. We define the induced map \linebreak $\iota_\phi \colon \Tr(T_{l_A}V^c_A) \to T_{l_A}V^c_A$ which maps $\ol{f}$ to $(-1)^{|i_\phi(f)''||i_\phi(f)'|}i_\phi(f)''i_\phi(f)'$. In particular, it maps $a_1\ten \cdots \ten a_n$ to
\[
\sum_i\pm \phi(a_i)''a_{i+1}\ten \cdots \ten a_n \ten a_1\ten \cdots \ten a_{i-1}\phi(a_i)'\: .
\]
From $\iota_\phi(w_A)=\del_\phi(\ol{w_A})$, we deduce that we have $f(\ol{\eta_B}^+(\phi))=(-1)^{|\phi||\ol{\eta_B}|}\iota_{\phi}(w_A)$. Since the element $\ol{\eta_B}$ is non-degenerate in $F_B\ten_{l_B^e}R_B$, the morphism $\ol{\eta_B}^+$ has an inverse \linebreak $\ol{\eta_B}^- \colon R_B \to \Hom_{l_B^e}(F_B,l_B^e)$ of degree $d-3$. If we apply $\phi=\ol{\eta_B}^-(v)$ to any homogeneous element $v$ of $R_B$, we obtain
\[
f(v)=(-1)^{(|v|+d-3)(d-3)}\iota_{\ol{\eta_B}^-(v)}(w_A)=(-1)^{(|v|+1)(d+1)}\iota_{\ol{\eta_B}^-(v)}(w_A)\: ,
\]
where we have
\[
v=\ol{\eta_B}^+(\phi)=(-1)^{|\phi||\ol{\eta_B}|}\phi(\ol{\eta_B}')''\ol{\eta_B}''\phi(\ol{\eta_B}')'\: .
\]
Since the restriction of $D$ to $V_B^c$ is injective, we may identify $V_B^c$ with its image under $D$ to write
\begin{align*}
2\iota_\phi(\omega_{\eta_B}) &=\iota_\phi ((D\eta_B')(D\eta_B'')) \\
                             &=\phi(\eta_B')''(D\eta_B'')\phi(\eta_B')'-(-1)^{|\eta_B'||\phi|}\phi(\eta_B'')''(D\eta_B')\phi(\eta_B'')' \\
                             &=\phi(\ol{\eta_B}')''(D\ol{\eta_B}'')\phi( \ol{\eta_B}')'+\phi(\ol{\eta_B}')''(D\ol{\eta_B}'')\phi( \ol{\eta_B}')' \\
                             &=2(-1)^{|\phi||\ol{\eta_B}|}Dv\: .
\end{align*}
Finally, we find that we have
\[
\iota_{\ol{\eta_B}^-(v)}(\omega_{\eta_B})=(-1)^{(|v|+d-3)(d-3)}Dv=(-1)^{(|v|+1)(d+1)}Dv\: .
\]
Since the element $\eta_B$ is non-degenerate and graded anti-symmetric in $V^c_B\ten_{l_B^e}V^c_B$, it yields a double Poisson bracket $\ldb?,-\rdb_{\omega_{\eta_B}}$ on $T_{l_A}(V^c_B\oplus N_A)$. So we have
\[
\ldb u,v\rdb_{\omega_{\eta_B}}=-(-1)^{(|u|-|\omega_{\eta_B}|)(|v|-|\omega_{\eta_B}|)}(-1)^{|\ldb u,v\rdb_{\omega_{\eta_B}}'||\ldb u,v\rdb_{\omega_{\eta_B}}''|}\ldb v,u\rdb_{\omega_{\eta_B}}''\ten \ldb v,u\rdb_{\omega_{\eta_B}}'\: .
\]
In particular, we have $|\ldb u,v\rdb_{\omega_{\eta_B}}'|=|\ldb v,u\rdb_{\omega_{\eta_B}}''|$ and $|\ldb u,v\rdb_{\omega_{\eta_B}}''|=|\ldb v,u\rdb_{\omega_{\eta_B}}'|$. Therefore, we have
\begin{align*}
f(v) &=\iota_{H_v}(w_A) \\
     &=\iota_{H_v}(Dw_A) \\
     &=(-1)^{|i_{H_v}(Dw_A)''||i_{H_v}(Dw_A)'|}i_{H_v}(Dw_A)''i_{H_v}(Dw_A)' \\
     &=(-1)^{|H_v(w_A)''||H_v(w_A)'|}\ldb v,w_A\rdb_{\omega_{\eta_B}}''\ldb v,w_A\rdb_{\omega_{\eta_B}}' \\
     &=\pm \ldb w_A,v\rdb_{\omega_{\eta_B}}'\ldb w_A,v\rdb_{\omega_{\eta_B}}''\: ,
\end{align*}
where the sign is given by the parity of
\[
1+|\ldb v,w_A\rdb_{\omega_{\eta_B}}''||\ldb v,w_A\rdb_{\omega_{\eta_B}}'|+(|v|+d-3)(|w_A|+d-3)+|\ldb w_A,v\rdb_{\omega_{\eta_B}}'||\ldb w_A,v\rdb_{\omega_{\eta_B}}''|\: .
\]
Thus, we have
\begin{align*}
f(v) &=-(-1)^{(|v|+d-3)(|w_A|+d-3)}\ldb w_A,v\rdb_{\omega_{\eta_B}}'\ldb w_A,v\rdb_{\omega_{\eta_B}}'' \\
     &=-\{w_A,v\}_{\omega_{\eta_B}}
\end{align*}
for all homogeneous elements $v$ of $R_B$. Since both sides are additive in $v$, the same equality holds for all $v$ in $R_B$.

Since the differential of $A$ squares to zero and $f$ commutes with
the differential, the equivalent conditions in Propositions~\ref{prop:necklace a1} and \ref{prop:necklace a2} hold. This concludes the proof of the implication from ii) to i)
in Theorem~\ref{thm:main}.

%\bibliographystyle{amsplain}
%\bibliography{stanKeller}

%\end{document}

\def\cprime{$'$} \def\cprime{$'$}
\providecommand{\bysame}{\leavevmode\hbox to3em{\hrulefill}\thinspace}
\providecommand{\MR}{\relax\ifhmode\unskip\space\fi MR }
\providecommand{\MRhref}[2]{%
  \href{http://www.ams.org/mathscinet-getitem?mr=#1}{#2}
}
\providecommand{\href}[2]{#2}

\end{document}